%% file: Oct-17-2018_LS_arxiv.tex
\documentclass[article,onefignum,onetabnum]{siamart171218}

\input{shared}

\ifpdf
\hypersetup{
  pdftitle={Primal dual mixed finite element
  methods for indefinite advection--diffusion equations},
  pdfauthor={E. Burman and C. He}
}
\fi

\begin{document}

\maketitle

% REQUIRED
\begin{abstract}
We consider primal-dual mixed finite element methods for
the advection--diffusion equation. For the primal variable we use
standard continuous finite element space and for the flux we use
the Raviart-Thomas space. We prove optimal a priori
error estimates in the energy- and the $L^2$-norms for the primal variable in the low Peclet
regime. 
In the high Peclet regime we also prove optimal
error estimates for the primal variable in the $H(div)$ norm for
smooth solutions. 
Numerically we observe that the method eliminates the
spurious oscillations close to interior layers that pollute the solution of the
standard Galerkin method when the local Peclet number is high.
This method, however, does produce spurious solutions when outflow
boundary layer presents. In the last section we propose two simple
strategies to remove such numerical artefacts caused  by the outflow
boundary layer and validate them numerically.
\end{abstract}

% REQUIRED
\begin{keywords}
Advection--Diffusion, Primal Dual Method, Mixed Finite Element Method
\end{keywords}

% REQUIRED
\begin{AMS}
  	65N30
\end{AMS}

\section{Introduction}
Advection--diffusion problems have been extensively studied in the last decades for its wide applications in the area of weather-forecasting, oceanography, gas dynamics, contaminant transportation in porous media, to name a few.
Many numerical methods for advection--diffusion equations have been
explored in the literature. The two main concerns when designing a numerical
method for advection--diffusion problems are robustness in the
advection dominated limit and satisfaction of local conservation. The standard
Galerkin method, using globally continuous approximation, is known to fail on both points and therefore
much effort has been devoted to the design of alternative formulations.
Typically to make the method stable in the
limit of dominating advection some stabilizing operator is introduced to provide sufficient
control of fine scale fluctuations. The most
well known stabilized method is the Streamline upwind Petrov
Galerkin method (SUPG) introduced by Hughes and
co-workers \cite{BH82} and first analyzed by Johnson and co-workers \cite{JNP84}. In order to
avoid disadvantages associated to the Petrov-Galerkin character, for
instance related to time discretization, the discontinuous Galerkin
method was introduced, first in the context of hyperbolic transport \cite{JP86,CS89}. In this case the stabilizing mechanism is due
to the upwind flux, which controls the solution jump over element
faces and adds a dissipation proportional to this jump. 
In the context
of finite element methods using $H^1$-conforming approximation several stabilized methods using
symmetric stabilization have been proposed, for instance the subgrid viscosity
method by Guermond \cite{Guer99}, the
orthogonal subscale method by Codina \cite{Cod00}, the continuous interior penalty method
(CIP), introduced by Douglas and Dupont \cite{DD76} and analyzed by Burman
and Hansbo \cite{BH04}. 
It is well known that for both cases of 
discontinuous and continuous approximation spaces a local conservative numerical flux can be defined. In the continuous case, however, it must be reconstructed using post processing
\cite{HEML00,BQS10}.

In this work, to ensure local conservation of the computed flux we design a method in the mixed setting: we approximate
the primal variable in the standard conforming finite element space and the flux in the
Raviart-Thomas space. The numerical scheme is based on a constrained
minimization problem in which the difference between the flux variable and the flux
evaluated using the primal variable is minimized under the constraint
of the conservation law. 
%This can be seen as a mixed variant of the stabilized
%finite element method for indefinite advection--diffusion problems
%introduced in \cite{Bu13}. Contrary to this work we herein consider a
%formulation where the approximation spaces are chosen so that it is
%inf-sup stable. 
The method is very robust and was initially introduced for the
approximation of ill-posed problems, such as the elliptic Cauchy
problem, see \cite{BLO17}. Herein we consider well-posed, but possibly indefinite
advection--diffusion equations. However, the results extend to
ill-posed advection--diffusion equations using the ideas of
\cite{BLO17} and \cite{BON18}.

Indefinite, or noncoercive, elliptic problems with Neumann boundary
conditions were considered first in \cite{CF88} and more recently in 
\cite{CD11,KT12, Bu13} using finite volume and finite element methods. 
The method proposed herein is a mixed variant of the primal-dual stabilized
finite element method introduced in \cite{Bu13,Bu14a} for the respective
indefinite elliptic and hyperbolic problems, drawing on
earlier ideas on $H^{-1}$-least square methods from \cite{BLP97}. 
Contrary to those work we herein consider a
formulation where the approximation spaces are chosen so that it is
inf-sup stable. Hence no stabilizing terms are required.
Primal dual methods without stabilization were proposed for
the advection--diffusion problem in \cite{CEQ14} and for second order PDE in
\cite{BQ17, BJ18}, inspired by previous work on Petrov-Galerkin
methods \cite{DG10, DHSW12}.
Similar ideas have recently been exploited successfully in the context of weak
Galerkin methods for elliptic problems on non-divergence form
\cite{WW18b}, Fokker-Planck equations \cite{WW17}, and the ill-posed
elliptic Cauchy problem in \cite{WW18a}.
In \cite{LWZ17} a method was introduced which is reminiscent of the lowest order version of the method we propose
herein. The case of high Peclet number was, however, not considered in
\cite{LWZ17}, so our analysis is likely to be useful for the
understanding of the method in \cite{LWZ17} in this regime.

%For the method discussed herein 
\subsection{Main results}
For the error analysis, in the low Peclet regime,
we prove optimal convergence orders for the $L^2$-  and
$H^1$- norms for the primal variable for all polynomial orders. For
the analysis we do not use coercivity, but only the stability of the solution, showing the interest of the method for indefinite (or
T-coercive \cite{Cia12}) problems. 
In the high Peclet regime we assume that the data of the adjoint
operator satisfy a certain positivity criterion, which is different to the
classical one for coercivity.
We then
prove an error estimate in negative norm and optimal order convergence of the error in the streamline
derivative of the primal variable measured in the $L^2$- norm, for
smooth solutions. 

Numerical results for both the diffusion- and advection-dominated
problems are presented. Optimal convergence is verified on smooth
problems and on a problem with reduced regularity due to a corner singularity.
We note that for problems with an internal layer only mild and localized oscillations 
are observed (see \cref{fig:osillation-free_2}). However, for problems
with under-resolved outflow boundary layers the effect of the layer
causes global pollution of the solution (see
\cref{fig:osillation}). In \cref{sec:outflow} we propose two simple
strategies to improve the method in this case. More specifically, one
method imposes the boundary condition weakly and the second approach
introduces a weighting of the stabilizer such that the oscillation is
more ``costly" closer to the inflow boundary. This latter variant introduces a notion of upwind direction.

This paper is organized as follows. In \cref{sec:problem}, the model problem is presented. The numerical scheme is proposed and its stability and continuity is analyzed in \cref{sec:mixed-method}. In \cref{sec:error-estimate}, we prove the error estimation results for both problems with either low or high Peclet numbers. Numerical results are presented in \cref{sec:numerical-results}. 
In \cref{sec:outflow} we propose two strategies to
  improve accuracy in the presence of underresolved outflow boundary layers. Numerical results are also presented to test their effectiveness.

\section{The Model Problem}\label{sec:problem}
Let $\Omega \in \mathbb{R}^d$, $d \in \{2,3\}$, be a polygonal/polyhedral domain,
with boundary $\partial \Omega$ and outward pointing unit normal $\bfn$. We consider the
following advection--diffusion equation,
\begin{equation}\label{CD_strong}
\nabla \cdot  ({\bf{\bfbeta}} u - A  \nabla u) + \mu u = f 
\end{equation}
with the boundary conditions
%The matrix $A$ is symmetric positive definite, with
%smallest eigenvalue $\lambda_{min,A}$.
%We deliberately leave the question of boundary conditions open at this
%stage, but will consider both Dirichlet and Neumann conditions.
%The method can however also handle indefinite advection-diffusion problems
%such as
\begin{equation}\label{CD_indef}
\begin{array}{rcl}
u&=& g  \mbox{ on } \Gamma_D, \quad \mbox{and}\\
(A \nabla u - {\bf{\bfbeta}} u) \cdot \bfn & = & \psi \mbox{ on } \Gamma_N.
\end{array}
\end{equation}
%with boundary condition $u=g$ on $\Gamma_D$ and $({\bf{\bfbeta}} u - A  \nabla u) \cdot \textbf{n} = \psi$ on $\Gamma_N$ 
where
  $\Gamma_D,\,\Gamma_N \subset \partial \Omega$, $\Gamma_D \cap \Gamma_N = \emptyset$ and $\bar \Gamma_D
\cup \bar \Gamma_N = \partial \Omega$. For simplicity, we assume that
$\Gamma_D \neq \emptyset$.
The problem data is given by $f \in
L^2(\Omega)$, $g \in
H^{\frac12}(\Gamma_D)$, $\psi \in
H^{-\frac12}(\Gamma_N)$, $A \in \mathbb{R}^{d\times d}$, $\mu \in
\mathbb{R}$ and $\bfbeta \in [L^\infty(\Omega)]^d$, with $\bfbeta_\infty := \|\bfbeta\|_{L^\infty(\Omega)}$. For the analysis in the
advection dominated case we will strengthen the assumptions on the parameters. Furthermore, we assume that
the matrix $A$ is symmetric positive definite.
With the smallest eigenvalue $\lambda_{min,A}>0$ and the largest
eigenvalue $\lambda_{\max,A}$ we assume that $\lambda_{\max,A}/\lambda
_{min,A}$ is bounded by a moderate constant. The below analysis holds
also in the case of variable $A$ and $\mu$, that are piecewise differentiable
on polyhedral subdomains, provided ajustements are made for loss of regularity
in the exact solution.

Let
\textcolor{black}{
\begin{equation}\label{W_V}
V_{g,D} \!=\! \{v \in H^1(\Omega): v=g \mbox{ on } \Gamma_D\}
\quad \mbox{and} \quad
V_{0,D}\!=\!\{v \in H^1(\Omega): v =0 \mbox{ on } \Gamma_D\}.
\end{equation}
}
Consider the weak form: find $u \in V_{g,D}$
 such that 
\begin{equation}\label{weak_CD}
a(u,v) = l(v), \quad \forall \, v \in V_{0,D},
\end{equation}
with
\[
a(u,v) := ( \mu u, v)_\Omega + (A \nabla u - \bfbeta u, \nabla v)_\Omega,
\]
and
\[
l(v) := (f,v)_\Omega  +  \left<\psi, v \right>_{\Gamma_N},
\]
where $(\cdot, \cdot)_w$ denotes the $L^2$ inner product on $w$. When
$w$ coincides with the domain $\Omega$ the subscript is omitted below.
%and $W,\, V$ are suitable subspaces of $H^1$ ($V=\{v \in H^1(\Omega): v =0 \mbox{ on } \Gamma_D\}$). 
We will only assume  that
the problem satisfies the Babuska-Lax-Milgram theorem \cite{babuska71}, which, in the case of homogenous Dirichlet condition, implies the
existence and uniqueness and 
the following stability estimate
\[
\|u\|_V \leq \alpha^{-1} \|l\|_{V'},
\]
where $\|\cdot\|_V$ is the $H^1$-norm, $\alpha$ is the constant of the inf-sup condition and 
the dual norm is defined by
\[
\|l\|_{V'} := \sup_{\substack{v \in V\\ \|v\|_V=1}} l(v).
\]
Observe that in the case of heterogenous Dirichlet condition we may
write $u = u_0+u_g$ where $u_0 \in V_{0,D}$ is unknwon and $u_g \in V_{g,D}$ is a
chosen lifting of the boundary data such that
$\|u_g\|_V \leq \|g\|_{H^{\frac12}(\Gamma_D)}$. In that case the
stability may be written as
\begin{equation}\label{Dirichlet_stab}
\|u_0\|_V \leq \alpha^{-1} \|l_g\|_{V'},
\end{equation}
where $l_g(v) = l(v) - a(u_g,v)$. 
\section{The Mixed Finite Element Framework}\label{sec:mixed-method}
\subsection{Some preliminary results}
%\section{A mixed least squares finite element method}
Let $\{\mathcal{T}\}_h$ be a family of conforming, quasi uniform
triangulations of $\Omega$
consisting of shape regular simplices $\mathcal{T} = \{K\}$. The
diameter of a simplex $K$ will be denoted by $h_K$ and the family index
$h$ is the mesh parameter defined as the largest diameter of all
elements, i.e,  $h = \underset{K \in \mathcal{T}}{\max} \{h_K\}$.
%For each simplex $K$ define the outward pointing normal by $\bfn_K$. 
%We assume that the boundary faces of
%\textcolor{black}{
%$\mathcal{T}$ fits $\partial \Omega$ so that $ \partial \Omega$}
%nowhere cuts through a boundary element face. 
We denote by $\mathcal{F}$  the set of all faces in $\mathcal{T}$,
by $\mathcal{F}_I$ the set of all interior faces in $\mathcal{T}$
and by $\cF_D$ and $\cF_N$ the sets of faces on
the respective $\Gamma_D$ and $\Gamma_N$.
For each $F \in \cF$ denote by $\bn_F$ a unit vector normal to $F$ and $\bn_F$ is fixed to be outer normal to $\partial \Omega$ when $F$ is a boundary face.
% whose union coincides with $\Gamma$ by $\mathcal{F}_\Gamma$.

Frequently, we will use the notation $a \lesssim b$ meaning $a \leq C b$
where $C$ is a non-essential constant, independent of $h$. Significant
properties of the hidden constant will be highlighted.

We denote the standard $H^1$-conforming finite
element space of order k by
\[
V_h^k:=\{v_h \in H^1(\Omega) : v|_K \in \mathbb{P}_k(K), \quad \forall \, K \in \cT \}
\]
where $\mathbb{P}_k(K)$ denotes the set of polynomials of degree less
than or equal to $k$ in the simplex $K$. 
Let $i_h: C^0(\bar\Omega) \mapsto V_h^k$ be the nodal interpolation.
The following
approximation estimate is satisfied by $i_h$, see e.g., \cite{EG04}. For $v \in
H^{k+1}(\Omega)$ there holds
\begin{equation}\label{eq:approx_Lagrange}
\|v - i_h v\| + h \|\nabla (v - i_h v)\|  \lesssim
h^{k+1} |v|_{H^{k+1}(\Omega)}, \quad k\ge 1.
\end{equation}

%defining the nodal interpolant $i_h:C^0(\bar \Omega) \mapsto V_h^k$, there holds {\color{black}$i_h:V_\Gamma \mapsto V_{g,D}^k$}. 

For the primal variable we introduce the following spaces
\[
V_{g,D}^k:=\{v_h \in V_h^k\, : \,v_h = g_h \mbox{ on } \Gamma_D \}
\quad \mbox{and} \quad
 V_{0,D}^k:=\{v_h \in V_h^k\, : \,v_h = 0 \mbox{ on } \Gamma_D \},
\]
where $g_h$ is the nodal interpolation of $g$ (or if $g$ has
insufficient smoothness, some other optimal approximation of $g$) on the trace of
%functions in $V_h $ on 
$\Gamma_D$ so that $g_h$ is piecewise polynomial of order $k$ with respect to $\cF_D$.

For the flux variable we use the Raviart-Thomas space
\[
RT^l:= \{\bfq_h \in H_{div}(\Omega)\, : \, \bfq_h\vert_K \in
\mathbb{P}_l(K)^d\oplus \bfx (\mathbb{P}_{l}(K) \setminus
\mathbb{P}_{l-1}(K)), \; \forall \,K \in \mathcal{T} \},
\]
with $\bfx \in \mathbb{R}^d$ being the spatial variable, $l\ge 0$ and
$\mathbb{P}_{-1}(K) \equiv \emptyset$.
We recall the Raviart-Thomas interpolant
$\bfR_h:H^1(div,\Omega) \mapsto RT^l$, where
\[
H^m(div,\Omega):= \left\{\bfw \in [H^m(\Omega)]^d \,:\, \nabla \cdot \bfw \in
H^m(\Omega) \right\},
\]
 and its approximation
  properties \cite{EG04}.  For \textcolor{black}{$\bfq \in
H^m(div,\Omega)$, $m \ge 1$ and $\bfR_h \bfq \in RT^{l}$, there holds
\begin{equation}\label{eq:approx_RT}
\|\bfq  - \bfR_h \bfq\|_\Omega + \|\nabla \cdot (\bfq  - \bfR_h
\bfq)\|_\Omega \lesssim h^{r} (|\nabla \cdot \bfq |_{H^{r}(\Omega)}+|\bfq |_{H^{r}(\Omega)})
\end{equation}
where $r=\min(m,l+1)$}.

We also introduce the
$L^2$-projection on the face $F$ of some simplex $K \in \mathcal{T}$, 
\[ \piF:L^2(F) \mapsto \mathbb{P}_{l}(F)\] 
such that for any $\phi \in L^2(F)$ %$\piF (\phi)$ satisfies
\[
\left<\phi- \piF (\phi), p_h \right>_F = 0,\quad \forall \, p_h \in \mathbb{P}_l(F).
\]
Then by assuming that
the Neumann data $ \psi$ is in $L^2(\Gamma_N)$ we define the discretized Neumann boundary data by its
$L^2$-projection such that for each $F \in \cF_N$ we have $ \psi_h\vert_F := \piF(\psi)$.
%A space for the flux variable, 
With the
satisfaction of the Neumann condition built in, we define
\[
RT_{\psi,N}^l=\{\bfq_h \in RT^{l} : \bfq_h \cdot \bfn = 
  -\psi_h \mbox{ on } \Gamma_N \}
  \]
  and
  \[
RT_{0,N}^l=\{\bfq_h \in RT^{l} : \bfq_h \cdot \bfn = 0 \mbox{ on } \Gamma_N \}.
\]

For the Lagrange multiplier variable, we introduce the space of
functions in $L^2(\Omega)$ that are piecewise polynomial of order $m$ in each
element by
\[
X^m_h := \{x_h \in L^2(\Omega): x_h\vert_K \in \mathbb{P}_m(K), \quad \forall \,K \in \mathcal{T}\}.
\]
We define the $L^2$-projection $\pi_{X,m}:L^2(\Omega) \mapsto X_h^m$ 
such that
\[
(y - \pi_{X,m}(y) , x_h) = 0,	\quad \forall  \,x_h \in X_h^m.
\]
For functions in $X^m_h$ we define the broken norms,
\begin{equation}\label{eq:discrete_H1}
\|x\|_h := \left(\sum_{K \in \mathcal{T}} \|x\|_K^2\right)^{\frac12}\mbox{
    and } \|x\|_{1,h} := \left(\|\nabla x\|^2_h +
  \|h^{-\frac12}\jump{x}\|^2_{\cF_I \cup \cF_D} \right)^{\frac12}
\end{equation}
where $\|h^{-1/2}x\|^2_{\mathcal{F}_I \cup \cF_D}:= \sum \limits_{F
  \in \mathcal{F}_I \cup \cF_D} h_F^{-1}\|x\|_F^2$ and
\[
\jump{x}\vert_F(z) := \left\{ \begin{array}{ll}
\lim\limits_{\epsilon
  \rightarrow 0^+} (x(z - \epsilon \bfn_F) - x(z + \epsilon \bfn_F))&
                             \mbox{ for } F \in \mathcal{F}_I,\\
x(z)& \mbox{ for } F \in \cF_D\cup \cF_N.
\end{array}
\right. 
\]
Also recall the discrete Poincar\'e inequality \cite{Brenner:03},

\begin{equation}\label{eq:disc_poincare}
\|x\|\lesssim \|x\|_{1,h}, \quad \forall \,x \in X_h^m,
\end{equation}
which guarantees that $\|\cdot\|_{1,h}$ is a norm. 

Given a function $x_h \in X_h^m$ we define a reconstruction,  $\bfeta_h(x_h) \in RT^l_{0,N}$,
of the gradient of $x_h$
%\TODO{The construction of $\eta_h(x_h)$ is really where magic happens. I would like to see the details.}
%By the properties of the Raviart-Thomas element there exists $\bfeta_h(x_h)\in D_{0}^l$ 
such that for all $F \in \mathcal{F}_I \cup \cF_D$ 
\begin{equation}\label{eq:reconstruct1}
\left<\bfeta_h(x_h) \cdot \bfn_F,  p_h \right>_F= \left< h_F^{-1} \jump{x_h},
p_h \right>_F, \quad \forall \, p_h \in \mathbb{P}_{l}(F),
\end{equation}
where $h_F$ is the diameter of $F$, and if $l \ge 1$, for all $K \in \mathcal{T}$,
\begin{equation}\label{eq:reconstruct2}
( \bfeta_h(x_h) ,\bfq_h)_K = -(\nabla x_h,
\bfq_h )_K,  \quad \forall \, \bfq_h \in [\mathbb{P}_{l-1}(K)]^d.
\end{equation}
%The stability of $\bfeta_h$ with respect to data is crucial in the
%analysis and we therefore prove it in a proposition.

We prove the stability of $\bfeta_h$ with respect to the data in the following proposition.
\begin{prop}
There exists an unique $\bfeta_h \in RT_{0,N}^l$ such that
\cref{eq:reconstruct1}--\cref{eq:reconstruct2} hold.
% for every face $F \in \mathcal{F} \setminus \mathcal{F}_\Gamma$ and every element in the mesh. 
Moreover $\bfeta_h$ satisfies the following stability estimate
\begin{equation}\label{eq:etastab}
\|\bfeta_h\| \leq C_{ds} \left(\|\pi_{X,l-1} \nabla x_h\|^2_{h} +
\|h^{-\frac12} \piF (\jump{x_h})\|^2_{\mathcal{F}_I \cup \cF_D}\right)^{\frac12},
\end{equation}
here $C_{ds}>0$ is a constant depending only on the element
shape regularity. 
\end{prop}

\begin{proof}
We refer to \cite{BLO17} for the proof.
\end{proof}

%To measure the effect of the perturbed data we introduce the corrector
%function $\delta \bfp \in D_{\delta \psi}^l$,
%\begin{equation}\label{eq:lift}
%\left<\delta \bfp\cdot \bfn_F , p_h\right>_F = \left\{ \begin{array}{ll}
% \left<\delta \psi,p_h\right>_F& \mbox{ for all } p_h \in
% \mathbb{P}_l(F) \mbox{ for } F \in \mathcal{F}_\Gamma\\
%0 & \mbox{ for all } p_h \in
% \mathbb{P}_l(F) \mbox{ for } F \in \mathcal{F} \setminus\mathcal{F}_\Gamma
%\end{array} \right.
%\end{equation}
%and if $k \ge 1$, for any $K \in \mathcal{T}$,
%\[
%(\delta \bfp, \bfq_h)_K = 0, \mbox{ for all } \bfq_h \in [\mathbb{P}_{l-1}(K)]^d.
%\]
%For $\delta \bfp$ we may also show the bound
%\begin{equation}
%\label{eq:deltap}
%\|\delta \bfp\|_\Omega \lesssim h^{\frac12} \|\delta \psi\|_\Gamma.
%\end{equation}

We will also frequently use the following inverse and trace
inequalities, 
\begin{equation}\label{eq:inverse}
\|\nabla v\|_K \lesssim h^{-1} \|v\|_K, \quad \forall \, v \in \mathbb{P}_k(K)
\end{equation}
and 
\begin{equation}\label{eq:trace}
\|v\|_{\partial K} \lesssim h^{-\frac12} \|v\|_K + h^{\frac12}\|\nabla v\|_K
, \quad \forall \, v \in H^1(K).
\end{equation}
For a proof of \cref{eq:inverse} we refer to Ciarlet \cite{CIARLET:78}, and
for \cref{eq:trace} see, e.g., Monk and S\"uli \cite{MS99}.

\subsection{The finite element method}
%Introducing the Lagrange multiplier space $W^m := X_h^m$. 
The problem takes the form of finding the critical point of a
Lagrangian
$\mathcal{L}: (v_h, \bfq_h, x_h) \in V_{g,D}^k \times RT_{ \psi,N}^l \times X_h^m \mapsto \mathbb{R}$ defined by
\begin{equation}\label{eq:lagrangian}
\mathcal{L}[v_h,\bfq_h,x_h]:= \frac12 s[(v_h,\bfq_h),(v_h,\bfq_h)] + b(\bfq_h, v_h,x_h)-(f,x_h).
\end{equation}
Here $x_h \in X_h^m$ is the Lagrange multiplier, $s(\cdot,\cdot)$
denotes the primal stabilizer
\begin{equation}\label{eq:primal_stab}
s[(v,\bfq),(v,\bfq)]:=  \frac12\|{\bf \bfbeta} v - A \nabla v -
\bfq\|^2,
\end{equation}
 and $b(\cdot,\cdot)$ is the
bilinear form defining the partial differential equation, in our case
the conservation law,
\[
b(\bfq_h,v_h,x_h):= (\nabla \cdot \bfq_h + \mu v_h, x_h).
\]

By computing the Euler-Lagrange equations of \cref{eq:lagrangian}
we obtain the following linear system: find 
$(u_h, \bfp_{h}, z_h) \in  V_{g,D}^k \times RT_{ \psi,N}^l \times X_h^m$ such that
\begin{align}\label{eq:EL_1}
s[(u_h,\bfp_h),(v_h,\bfq_h)]+b(\bfq_h, v_h,z_h) & = 0, \\
b(\bfp_h, u_h,x_h) - ( f,x_h)  & = 0\label{eq:EL_2},
\end{align}
for all $(v_h, \bfq_h ,x_h) \in  V_{0,D}^k \times RT_{0,N}^l \times X_h^m$.
The system \cref{eq:EL_1}--\cref{eq:EL_2} is of the
same form as that proposed in \cite{Bu14b,Bu16} but without the adjoint stabilization. 
Therefore, to ensure that the system is well-posed the spaces $V_h^k \times
RT^l \times X_h^m$ must be carefully balanced. Herein we will restrict the discussion to
the equal order case $k=l=m$ that is stable without further
stabilization. The arguments can be extended to other choices of
spaces provided suitable extra stabilizing terms are added (see \cite{BLO17} for details).

Observe that the stabilizer in equation \cref{eq:primal_stab} connects the flux and the primal variables and, more precisely, brings
$\bfp_h$ and ${\bf{\bfbeta}} u_h - A\nabla u_h$ to be close.
In the low Peclet regime this introduces an
effect similar to the penalty on the gradient of the primal
variable used in \cite{Bu13}. In the high Peclet regime, on the other
hand, the stability of the streamline derivative is obtained by the
strong control of the conservation law residual obtained through
equation \cref{eq:EL_2}.
\begin{remark}
The constrained-minimization problem introduces an auxiliary
variable, i.e., the Lagrange multiplier, which for stability reasons
must be chosen as the discontinuous counterpart of the discretization
space for the primal variable (unless stabilization is applied, see \cite{BLO17}). This results in a system with a
substantially larger number of degrees of freedom than the standard
Galerkin and the classical mixed method. Nevertheless, it is
possible to
reduce the system used in the iterative solver to a positive definite
symmetric matrix where the Lagrange multiplier has been
eliminated. This is achieved by
iterating on a least square formulation and the solution of which is not
locally mass conservative but has similar approximation
properties. The number of degrees of freedom of the reduced system is comparable  to that of the
mixed method using the Raviart-Thomas element. For a detailed discussion of
this approach we refer to \cite{BLO17}. 
\end{remark}
\subsection{Approximation, continuity and inf-sup condition}
For the analysis we introduce the energy norms on
$H^1(\Omega) \times H(div, \Omega)$,
\begin{align}
\tnorm{(v,\bfq)}_{-1} &:= \left(
s[(v,\bfq),(v,\bfq)] + \|h(\nabla \cdot
\bfq+ \mu v)\|^2 \right)^{\frac12}, \label{eq:tnorm}
\\
\tnorm{(v,\bfq)}_{\sharp} &:=\tnorm{(v,\bfq)}_{-1}+\|\mu
                            v\|+\|h^{\frac12}
                            \bfq\|_{\mathcal{F}}+\|\bfq\|_{\Omega}. \label{eq:cont_norm}
\end{align}
%where, depending on the choice of the spaces and stabilizers, either $\zeta = 0$ or $\zeta = -1$.
To quantify the dpendence of the physical parameters in the bounds
below we introduce the factor $c_u :=   \bfbeta_{\infty} h + \|A\|_{\infty} + |\mu|h $.
\begin{lem} [Approximation]\label{lem:approx}
For any $v \in H^{k+1}(\Omega)$ and $\bfq \in H^{l+1}(\Omega)^d$ the following 
approximation properties hold:
\begin{equation}\label{eq:sharpnorm_approx}
\tnorm{(v - i_h v,\bfq - R_h \bfq)}_{-1}  \le \tnorm{(v - i_h
  v,\bfq - R_h \bfq)}_{\sharp} \lesssim
  c_u h^k|v|_{H^{k+1}(\Omega)}+ h^{l+1} |\bfq|_{H^{l+1}(\Omega)}.
  \end{equation}
\end{lem}
\begin{proof}
Applying the triangle inequality and the approximation properties \cref{eq:approx_Lagrange} and
\cref{eq:approx_RT} gives
\begin{equation}\label{eq:tnorm_approx}
\begin{split}
	\tnorm{(v - i_h
  v,\bfq - R_h \bfq)}_{-1} 
%  \le&
%  \| \bfq - R_h \bfq\| + \|\bfbeta\|_{\infty}\|(v - i_h v\| + \|A\|_{\infty} \|\nabla (v - i_h v)\| \\
% 	+& 
%  h\| \nabla \cdot (\bfq - R_h \bfq)\| + h \|\mu\|_{\infty} \|v - i_h v\| \\
  \lesssim&
     \left(\bfbeta_{\infty} h + \|A\|_{\infty} + |\mu| h^{2} \right) h^k |v|_{H^{k+1}(\Omega)}
     +
  h^{l+1} |\bfq|_{H^{l+1}(\Omega)}.
\end{split}
\end{equation}
To estimate the remaining terms note that the trace inequality \cref{eq:trace} implies
\[
	\|h^{1/2} (\bfq - R_h \bfq)\|_\mathcal{F} \lesssim
	\|\bfq - R_h \bfq\| + h \|\nabla(\bfq - R_h \bfq)\|,
	%\lesssim \| \bfq - R_h \bfq\| + \|h\nabla \cdot (\bfq - R_h \bfq)\|
\]
which, combining with the approximation properties, 
gives
\[
	\|\mu (v - i_h v)\|+\|h^{\frac12} (\bfq -R_h\bfq) \|_{\mathcal{F}}+\|\bfq-R_h\bfq\|
	\lesssim
	| \mu| h^{k+1} |v|_{H^{k+1}(\Omega)} + h^{l+1} | \bfq|_{H^{l+1}(\Omega)}.
\]
\cref{eq:sharpnorm_approx} is then a direct consequence of the above inequality and \cref{eq:tnorm_approx}.
This completes the proof of the lemma.
\end{proof}

To facilitate the analysis we rewrite the system \cref{eq:EL_1}--\cref{eq:EL_2} in the following compact form:
 finding $(u_h,\bfp_h,z_h) \in V_{g,D}^k \times
RT_{\psi,N}^l \times X_h^m$ such that
\begin{equation}\label{eq:compact}
\mathcal{A}[(u_h,\bfp_h,z_h),(v_h,\bfq_h,x_h)] = l_h(x_h),\quad  \forall \,
(v_h,\bfq_h,x_h) \in V_{0,D}^k \!\times\! RT_{0,N}^l \!\times\! X_h^m,
\end{equation}
where
$$
\mathcal{A}[(u_h,\bfp_h,z_h),(v_h,\bfq_h,x_h)] = b(\bfq_h, v_h,z_h)+
b(\bfp_h, u_h,x_h)+s[(u_h,\bfp_h),(v_h,\bfq_h)],
$$
and
\[
 l_h(x_h) = ( f, x_h).
\]
Note that for the exact solution, $(u,\bfp)$, there
holds 
\begin{equation}\label{eq:consist_full}
\mathcal{A}[(u,\bfp,0),(v_h,\bfq_h,x_h)] = l(x_h).
\end{equation}
\begin{prop}[Inf-sup Condition]\label{prop:infsup}
 Let $k= l = m$ in \cref{eq:compact}.  Then there
exists $\alpha_c>0$ such that,
for all $(v_h,\bfq_h,x_h) \in V_{0,D}^k \times RT^k_{0,N} \times X_h^k$, there
exists $(\tilde v_h,\tilde \bfq_h, \tilde x_h) \in V_{0,D}^k \times RT^k_{0,N} \times
X_h^k$ satisfying
\begin{equation}\label{eq:is1}
\alpha_c (\tnorm{(v_h,\bfq_h)}_{-1}^2 + \|x_h\|_{1,h}^2) \leq \mathcal{A}[(v_h,\bfq_h,x_h),(\tilde v_h,\tilde \bfq_h, \tilde x_h)]
\end{equation}
and
\begin{equation}\label{eq:is2}
\tnorm{(\tilde v_h,\tilde \bfq_h)}_{-1} + \|\tilde x_h\|_{1,h} \lesssim \tnorm{(v_h,\bfq_h)}_{-1} + \|x_h\|_{1,h}.
\end{equation}
\end{prop}
\begin{proof}
Define $\bfeta_h = \bfeta_h(x_h) \in RT_{0,N}^{k}$ by taking $l=m=k$ in \cref{eq:reconstruct1}--\cref{eq:reconstruct2} and $\xi_h :=h^2 (\nabla
\cdot \bfq_h + \mu v_h) \in X_h^k$.
We claim that, by choosing $ \tilde v_h = v_h \in V_{0,D}^k,\, \tilde \bfq_h  = \bfq_h+
\epsilon \bfeta_h \in RT_{0,N}^k$ and $\tilde x_h = -x_h+ \xi_h \in X_h^k$,
 there holds \cref{eq:is1} and \cref{eq:is2},
where $\epsilon$ is to be determined later.

By the above definitions, we have
\begin{equation}  \label{stability_1}
\begin{split}
	&\mathcal{A}[(v_h,\bfq_h,x_h),(v_h,\bfq_h+ \epsilon
        \bfeta_h,-x_h+ \xi_h)]
\\ 
=&\,  \|\bfq_h - \bfbeta v_h + A \nabla
        v_h\|^2   + \| h ( \nabla \cdot \bfq_h+ \mu v_h) \|^2\\
&+  
	\epsilon(\bfq_h - \bfbeta v_h + A \nabla v_h,  \bfeta_h )
	+	\epsilon( \nabla \cdot  \bfeta_h , x_h).
\end{split}
\end{equation} 
For the last term, it follows from integration by parts, \cref{eq:reconstruct1}, \cref{eq:reconstruct2}
and the facts that $\bfeta_h \cdot \bfn = 0$ on $\Gamma_N$, $\nabla x_h|_K \in \mathbb{P}_{k-1}(K)^d$ and ${x_h}|_F \in \mathbb{P}^k(F)$ that
\begin{eqnarray*}
	( \nabla \cdot  \bfeta_h , x_h) &=& \sum_{K \in \cT} \left(
	-( \bfeta_h , \nabla x_h)_K  + \left<  \bfeta_h \cdot \bfn_K, x_h \right>_{\partial K}\right)=
%	\\
%	&=& 
	\| \nabla x_h\|^2 + 
	\sum_{F\in \mathcal{F}_I \cup \cF_D}\|h^{-\frac12}
	 \jump{x_h}\|^2_{F},
\end{eqnarray*}
which, combining with \cref{stability_1}, the Cauchy-Schwartz inequality and \cref{eq:etastab}, gives
\begin{equation} \label{stability_2}
\begin{split}
	&\mathcal{A}[(v_h,\bfq_h,x_h),(v_h,\bfq_h+ \epsilon \bfeta_h,-x_h+ \xi_h)]  \\
	 \ge&\,
	\|\bfq_h - \bfbeta v_h + A \nabla v_h\|^2  + \| h ( \nabla \cdot \bfq_h+ \mu v_h) \|^2
	- \dfrac{1}{4} \|\bfq_h - \bfbeta v_h + A \nabla v_h\|^2 \\
	&-\epsilon^2\|\bfeta_h\|^2 	\!+\! \epsilon\left( \| \nabla x_h\|^2 \!+\! 
	\sum_{F\in \mathcal{F}_I \cup \cF_D}\|h^{-\frac12} \jump{x_h}\|^2_{F} \right)\\
	\ge&\,
	\dfrac{3}{4}\|\bfq_h \!-\! \bfbeta v_h \!+\! A \nabla v_h\|^2  \!+\! \| h ( \nabla \cdot \bfq_h+ \mu v_h) \|^2
	\!+\!\epsilon(1 - \epsilon C_{ds}^2) \|x_h\|_{1,h}^2.
	\end{split}
	\end{equation}
 \cref{eq:is1} is then a direct result of \cref{stability_2} by choosing 
	$\epsilon = \dfrac{1}{2} C_{ds}^{-2}$ and
	$\alpha_c = \min\left(\dfrac{3}{4},\dfrac{1}{2}\epsilon \right)$.
	
To prove \cref{eq:is2} first applying the triangle inequality gives
\begin{equation}\label{coecive:1}
\tnorm{(\tilde v_h, \tilde \bfq_h)}_{-1} \!+\! \|\tilde x_h\|_{1,h} 
\leq
\tnorm{(v_h,\bfq_h)}_{-1} \!+\! \|x_h\|_{1,h}+\tnorm{(0,\epsilon \bfeta_h)}_{-1} \!+\! \|\xi_h\|_{1,h},
\end{equation}
%\begin{eqnarray*} \label{coecive:1}
%&& \tnorm{(\tilde v_h, \tilde \bfq_h)}_{-1} \!+\! \|\tilde x_h\|_{1,h} 
%\leq
%\tnorm{(v_h,\bfq_h)}_{-1} \!+\! \|x_h\|_{1,h}+\tnorm{(0,\epsilon \bfeta_h)}_{-1} \!+\! \|\xi_h\|_{1,h}.
%\end{eqnarray*}
%To bound the second to last term on the right hand side,
Then applying the trace and inverse inequalities  and \cref{eq:etastab} yields
\begin{equation} \label{coecive:2}
\tnorm{(0,\epsilon \bfeta_h)}_{-1} = \epsilon ( \|\bfeta_h\|
+\|h \nabla \cdot \bfeta_h\|) \lesssim
\|\bfeta_h\|  \lesssim \|x_h\|_{1,h},
\end{equation}
and
\begin{equation} \label{coecive:3}
\|\xi_h\|_{1,h}  \lesssim h^{-1}\| \xi_h \| = \|h (\nabla \cdot
\bfq_h + \mu v_h)\| \le \tnorm{v_h, \bfq_h}_{-1}.
\end{equation}
Combining \cref{coecive:1}--\cref{coecive:3} results in \cref{eq:is2}.
This completes the proof of the proposition.
%Since this proves that $\tnorm{(w_h,\bfy_h)}_{-1} + \|r_h\|_{1,h}
%\lesssim \tnorm{(v_h,\bfq_h)}_{-1} + \|w_h\|_{1,h}$ the proof is complete.
\end{proof}
\begin{prop}[Existence and Uniqueness]
The linear system defined by \cref{eq:compact}
admits an unique solution $(u_h ,\bfp_h,z_h) \in V_{g,D}^k \times
RT_{\psi,N}^k \times X_h^k$.
\end{prop}
\begin{proof}
 In order to prove the invertibility of the square linear system it is equivalent to prove the uniqueness. 
Assume that there exist two sets of solutions, 
$(u_{1,h},\bfp_{1,h},z_{1,h})$ and $(u_{2,h},\bfp_{2,h},z_{2,h})$, both in $ V_{g,D}^k \times
RT_{\psi,N}^k \times X_h^k$. We then have that for all $(v_n, \bfq_h, x_h)$ in  the space of $V_{0,D}^k \times RT^k_{0,N} \times X_h^k$ there holds
\[
\mathcal{A}[(u_{1,h} - u_{2,h},\bfp_{1,h} - \bfp_{2,h}, z_{1,h} - z_{2,h}),(v_h,\bfq_h,x_h)] = 0.
\]
By \cref{prop:infsup}, the following must be true:
\[ 
	\| (u_{1,h} - u_{2,h}, \bfp_{1,h} - \bfp_{2,h})\|_{-1} + \| z_{1,h} - z_{2,h}\|_{1,h}=0,
\]
which immediately implies
\[
	z_{1,h} = z_{2,h}  \quad \mbox{and} \quad
	\nabla \cdot \left(\bfbeta (u_{1,h} - u_{2,h})- A \nabla (u_{1,h} - u_{2,h})\right) + \mu (u_{1,h} - u_{2,h}) = 0.
\]
Since \cref{CD_strong}--\cref{CD_indef} admits an unique trivial solution for zero datum we conclude that $u_{1,h} = u_{2,h}$ and, hence,
$\bfp_{1,h} = \bfp_{2,h}$. This completes the proof of the proposition.
\end{proof}

We end this section by proving the continuity of the bilinear form.
\begin{prop}[Continuity]\label{prop:continuity}
For all $(v,\bfq) \in H^1(\Omega) \times
H_{0,N}(div,\Omega)$ and for all $(v_h,\bfq_h,x_h) \in V_h^k \times
RT^l \times X_h^m$  there holds
\begin{equation}\label{eq:cont1}
\mathcal{A}[(v,\bfq,0), (v_h,\bfq_h,x_h)] \leq \tnorm{(v,\bfq)}_{\sharp} \, (\tnorm{(v_h,\bfq_h)}_{-1}+\|x_h\|_{1,h}).
\end{equation}

\end{prop}
\begin{proof}
The inequality \cref{eq:cont1} follows by first using the Cauchy-Schwarz
inequality in the symmetric part of the formulation,
\[
s[(v,\bfq),(v_h,\bfq_h)]
 \lesssim s[(v,\bfq),(v,\bfq)]^{\frac12} s[(v_h,\bfq),(v_h,\bfq)]^{\frac12}.
\]
For the remaining term we use the divergence formula elementwisely to obtain
\[
(\nabla \cdot \bfq + \mu v, x_h) = \sum_{K \in \mathcal{T}}
-(\bfq,\nabla x_h)_K + \sum_{F \in \mathcal{F}_I \cup \cF_D} \left<\bfq \cdot \bfn_F,\jump{x_h}
\right>_F+(\mu v,x_h).
\]
\cref{eq:cont1} then follows by applying the Cauchy-Schwartz inequality and \cref{eq:disc_poincare}.
This completes the proof of the proposition.
\end{proof}

\section{Error Estimation} \label{sec:error-estimate}
In this section we will prove optimal error estimates for smooth solutions,
both in the diffusion dominated and the advection dominated
regimes. When the diffusion dominates we prove optimal error
estimates in both the $H^1$- and $L^2$-norms under very mild \textcolor{black}{stability} assumptions on
the continuous problem. In this part constants may blow up as the
Peclet number becomes large.

For dominating advection we need to make an
assumption on the problem data to prove an error estimate in the
$H^{-1}$-norm.  This is then used to prove an estimate that is optimal for the error in
the divergence of the flux, computed using the primal variable, or the ``streamline derivative''. However,
we can not improve on the order for the $L^2$-error as for typical
residual based stabilized finite element methods. In this part
constants remain bounded as the Peclet number becomes high.

\subsection{Error estimate for the residual}
First we prove the optimal convergence result for the residual,
%that the residual quantities will converge with optimal order. 
i.e., the optimal convergence for the triple norm
\cref{eq:tnorm}. This estimate will then be of use in both the high
and low Peclet regimes.
%\textcolor{black}{(remove)
%Let 
%\[ 
%	H^{l+1}_{\psi,N}(\Omega)^d \!=\! \{ \bfq \in H^{l+1}(\Omega)^d  : 
%	\bfq \cdot \bn|_{\Gamma_N} = \psi_N \}
%	\, \mbox{and} \,
%	H^{k+1}_{g,D}(\Omega) \!=\! \{ v \in H^{k+1}(\Omega)  :
%	v|_{\Gamma_D} = g_D \}.
%\]}

\begin{lem}[Estimate of Residual]\label{lem:residualest}
Assume that $(u,\bfp)$ is the solution to \cref{weak_CD} with 
$u \in H^{k+1} \cap V_{g,D}(\Omega)$, $\bfp \in  H^{l+1}(\Omega)^d \cap H_{\psi,N}(\Omega)^d$ and $l \leq k$, and
that $(u_h,\bfp_h,z_h) \in V_{g,D}^{k} \times RT_{\psi,N}^k \times X_h^k$ is the solution of
\cref{eq:compact}. 
Then there holds
\begin{equation} \label{Residual-estimate}
\tnorm{(u-u_h,\bfp - \bfp_h)}_{-1} +\|z_h\|_{1,h} \lesssim 
c_u h^k|u|_{H^{k+1}(\Omega)}
+ h^{l+1} |\bfp|_{H^{l+1}(\Omega)}.
\end{equation}

 \end{lem}
\begin{proof}
Firstly, applying the triangle inequality gives
  \begin{equation}\label{resi-error-1}
 	\tnorm{(u-u_h,\bfp - \bfp_h)}_{-1} \le 
	\tnorm{(u-i_h u,\bfp - R_h\bfp)}_{-1}+
	\tnorm{(u_h-i_h u,\bfp_h - R_h \bfp)}_{-1}.
  \end{equation}
 Note that $ u _h - i_h u \in V_{0,D}^k$ and $\bfp_h - R_h \bfp \in RT_{0,N}^k$. Then by \cref{prop:infsup} there exists 
 $(v_h, \bfq_h, w_h) \in  V_{0,D}^{k} \times RT_{0,N}^k \times X_h^k$ such that
% \[
% 	\tnorm{(u_h-i_h u,\bfp_h - R_h\bfp)}_{-1}^2 + \|z_h\|_{1,h}^2
%	\lesssim \mathcal{A}[(u_h-i_h u,\bfp_h - R_h\bfp, z_h)),(v_h, \bfq_h, w_h) ].
%  \] 
\textcolor{black}{
 \begin{equation*}
 \begin{split}
	& \tnorm{(u_h-i_h u,\bfp_h - R_h\bfp)}_{-1}^2 + \|z_h\|_{1,h}^2\\
	\lesssim& \, \mathcal{A}[(u_h-i_h u,\bfp_h - R_h\bfp, z_h)),(v_h, \bfq_h, w_h) ] 
%	&=&
%	  \mathcal{A}[(u_h,\bfp_h, z_h)),(x_h, \bfy_h, r_h) ]
%	  -
%	  \mathcal{A}[(i_h u, R_h\bfp, 0),(x_h, \bfy_h, r_h) ]\\
%	 &=&
%	   \mathcal{A}[(u_h,\bfp_h, z_h)),(x_h, \bfy_h, r_h) ]+
	= \mathcal{A}[(u-i_h u, \bfp-R_h\bfp, 0),(v_h, \bfq_h, w_h)  ]  \\
	%- 
	 %\mathcal{A}[(u , \bfp, 0),(x_h, \bfy_h, r_h) ] \\
%	 &=&
%	 \mathcal{A}[(u-i_h u, \bfp-R_h\bfp, 0),(v_h, \bfq_h, w_h) ] \\
	 \lesssim&\,
	 \tnorm {(u-i_h u, \bfp-R_h\bfp )}_{\sharp} 
	 \left(\tnorm{(u_h-i_h u,\bfp_h - R_h\bfp)}_{-1} + \|z_h\|_{1,h} \right).
\end{split}
 \end{equation*}
 }
The last equality and inequality follows from  \cref{eq:compact},  \cref{eq:consist_full} and \cref{prop:continuity}.
 Therefore, we immediately have that
\[
  	\tnorm{(u_h-i_h u,\bfp_h - R_h\bfp)}_{-1} \lesssim
	\tnorm {(u-i_h u, \bfp-R_h\bfp )}_{\sharp} 
\]
 which, combing with \cref{eq:sharpnorm_approx} and \cref{resi-error-1}, implies \cref{Residual-estimate}. This completes the proof of the lemma.
\end{proof} 

Observe that the hidden constant in \cref{Residual-estimate} has no
inverse powers of the diffusivity. Hence we have the following corollary.
\begin{cor}\label{col:residual-estimate-for-advection}
Under the same assumptions as in \cref{lem:residualest}, if
$\|A\|_\infty << h$, $\bfbeta_\infty = O(1)$, $|\mu| = O(1)$, there holds
\begin{equation} \label{Residual-estimate_HiPec}
\tnorm{(u-u_h,\bfp - \bfp_h)}_{-1}+ \|z_h\|_{1,h} \lesssim
h^{k+1} |u|_{H^{k+1}(\Omega)}+ h^{l+1} |\bfp|_{H^{l+1}(\Omega)}.
\end{equation}
\end{cor}
\subsection{Error estimates in the diffusion dominated regime}
In this subsection we provide results for the error estimation in the diffusion dominated regime, i.e.,
$\dfrac{ \bfbeta_\infty}{\lambda_{min,A}}$ is of order $1$ where $\lambda_{min,A}$ is the smallest eigenvalue of $A$.

\begin{prop}[$H^1$-norm estimate] \label{thm:condstab_est} Assume that $(u,\bfp)$ is the solution to \cref{weak_CD}, $u
\in  H^{k+1}(\Omega) \cap V_{g,D}(\Omega)$ and $\bfp \in  H^{l+1}(\Omega)^d \cap H_{\psi,N}(\Omega)^d$ with $l < k$, and that
$(u_h,\bfp_h)$ is the solution of \cref{eq:compact}. Then the following estimate holds,

\begin{equation}\label{eq:loc_error}
\|u - u_h\|_{V}  \le C \left(  c_u 
 h^k|u|_{H^{k+1}(\Omega)} + h^{l+1} \left( |\bfp|_{H^{l+1}(\Omega)} +
                   |\psi|_{H^{l+1/2}(\Gamma_N)} \right) \right),
\end{equation}
where the constant $C$ depends on the datum in the following manner
\[
	C \equiv \dfrac{\|\bfbeta\|_{\infty}}{\lambda_{min,A}}.
\]%\[
%	C \sim \dfrac{\|\bfbeta\|_{\infty}}{\lambda_{min,A}}.
%\]

\end{prop}

\begin{rem}
%The constant $C$ in the above proposition depend on the datum in the following manner
%\[
%	C \equiv \dfrac{\|\bfbeta\|_{\infty}}{\lambda_{min,A}}.
%\]
Since the above constant $C$ blows up as $\lambda_{min,A}$ goes to zero, the above estimation is valid only for diffusion dominated problem.
\end{rem}

\begin{proof}
To avoid using coercivity arguments, our starting point for the error analysis
below is the 
stability estimate \cref{Dirichlet_stab}. Let $e = u-u_h$. we note that $e$ is a solution to \cref{weak_CD} with the right hand side linear operator  being $r(v):=l(v) - a(u_h,v) $, i.e.,
\[
a(e,v) = r(v).
\]
  Now apply the decomposition $e = e_0 + e_g$ such that 
$e_g|_{\Gamma_D} = e|_{\Gamma_D}$ and that $\|e_g\|_{V} \lesssim \|e\|_{H^{1/2}(\Gamma_D)}$.
   It then follows from \cref{Dirichlet_stab} that
\[
\|e_0\|_V \leq C \sup_{v \in V} \dfrac{r(v) - a(e_g,v)}{\|v\|_{V}}
\leq C(\|r\|_{V'}+\|e_g\|_V).
\]
Hence 
\[
\|e\|_V \leq \|e_0\|_V +\|e_g\|_V \leq C (\|r\|_{V'}+\|e_g\|_V).
\]
For the term $\|e_g\|_V$, by definition and a standard
trace inequality, we have
\begin{equation}\label{e-g estimate}
\|e_g\|_V \leq C \|u - i_h u\|_{H^{\frac12}(\Gamma_D)} \leq C \|u -
i_h u\|_V \leq C h^k|u|_{H^{k+1}(\Omega)}.
\end{equation}
To prove the bound on $\|r\|_{V'}$ we recall that
\[
\|r\|_{V'} = \sup_{v \in V} \dfrac{a(u - u_h, v) }{\|v\|_{V}}.
\]
Then by integration by parts, (\ref{eq:EL_2}) and Cauchy Schwartz inequality, we have
	\begin{eqnarray*}
	&&a(u - u_h,v) = l(v) - a(u_h,v)\\
        & = & (f,v)+(\psi,v)_{\Gamma_N}
           - (A \nabla u_h - \beta u_h,\nabla v) - (\mu u_h,v) \\
	&=&(f - \nabla \cdot \bfp_h - \mu u_h, v)+(\psi+\bfp_h \cdot
            n, v)_{\Gamma_N} -
	( \bfp_h - \beta u_h +A \nabla u_h, \nabla v) \\
	&=&(f - \nabla \cdot \bfp_h - \mu u_h, v - \pi_{X,0}v)
	+
	(\bfp - \bfp_h - \beta (u -u_h) +A \nabla (u-u_h), \nabla v)\\
	&&+(\psi - \psi_h, v - \pi_{F,0} v)_{\Gamma_N}\\
	&\lesssim&
	\tnorm{(\bfp - \bfp_h),( u-u_h) }_{-1} \| \nabla v\| +
                   \|h^{\frac12} (\psi
                   - \psi_h)\|_{\Gamma_N} \| \nabla v\|
                   	\end{eqnarray*}
		which, combining with \cref{e-g estimate}, (\ref{Residual-estimate}) and the following
                observation (see e.g., Lemma 5.2 of \cite{Ern:2017})
\[
\|h^{\frac12} (\psi
                   - \psi_h)\|_{\Gamma_N} 
               \lesssim h^{l+1}
                   |\psi|_{H^{l+1/2}(\Gamma_N)} ,
\] gives 
\cref{eq:loc_error}. This completes the proof of the proposition.
\end{proof}

In the remaining part of this subsection we will focus on the
convergence of the $L^2$-norm error in the primal variable.
For
simplicity we here restrict the discussion to the case of a convex
polygonal domain $\Omega$ and
homogeneous Dirichlet condition.  We first prove the convergence result for the $L^2$-norm of the
Lagrange multiplier.

\begin{prop}\label{prop:L2-langrange}
Assume that $ u \in H_{0}^1(\Omega) \cap H^{k+1}(\Omega)$ and $\bfp \in H^{k}(\Omega)^d$. Let
 $z_h$ be the Lagrange multiplier of the system  {\em \cref{eq:compact}}. We have the following error estimate
 \begin{equation} \label{L2-multiplier}
 	\|z_h\|_{\Omega} \lesssim 
	 h^{k+1}  \left( |u|_{H^{k+1}(\Omega)}
	+  |\bfq|_{H^{k}(\Omega)} \right).
 \end{equation}
\end{prop}
\begin{proof}
Let $\phi$ be the solution such that
\[
	\nabla \cdot (\bfbeta \phi -A \nabla \phi) + \mu \phi = z_h
\]
with boundary condition $\phi = 0 $ on $\partial \Omega$.
Then by the well-posedness assumption on the equation
\cref{CD_strong} and the assumption on $\Omega$ we have the following stability result:

\begin{equation}\label{L2-stability}
	\| \phi\|_{H^2(\Omega)} \lesssim \|z_h\|. 
\end{equation}
Let $\bfq = \bfbeta \phi - A \nabla \phi$. By adding and subtracting suitable interpolates we have
\begin{equation}\label{L2_estimate_0}
	\|z_h\|^2 %= (z_h, \nabla \cdot \bfq + \mu \phi) 
	=
	(z_h, \nabla \cdot (\bfq -R_h \bfq) + \mu (\phi -i_h \phi))
	+(z_h, \nabla \cdot R_h \bfq + \mu i_h \phi).
\end{equation}
For the first term in \cref{L2_estimate_0} using the element-wise divergence theorem, the facts that
\[	
	\int_{F} z_h (\bfq - R_h \bfq) \cdot  \bn_K \,ds = 0, \quad \forall \, K \in \cT, 
	\, \forall F \subset \partial K,
\]
and that
\begin{equation*}%\label{L2-a}
	\| \bfq - R_h \bfq\| \lesssim h \| \phi\|_{H^2(\Omega)} \lesssim h \|z_h\| \quad \mbox{and} \quad
	\|\phi - i_h \phi\| \lesssim h^2 \|\phi\|_{H^2(\Omega)} \lesssim h^2 \|z_h\|,
\end{equation*}
and (\ref{eq:disc_poincare})
gives
\begin{equation} \label{L2-1}
 (z_h, \nabla \cdot (\bfq -  R_h \bfq) + \mu (\phi -i_h \phi))\\
% =&
% -\sum_{K \in \cT} (\nabla z_h, \bfq - R_h \bfq)_K +( \mu (\phi -i_h \phi), z_h) \\
\lesssim
 h (1 + \|\mu\|_{\infty}h)\|z_h\|_{1,h} \| z_h\|.
\end{equation}
 For the second term in \cref{L2_estimate_0}  we first apply equation \cref{eq:EL_1} with 
 $\bfq_h =
R_h (\bfq) \in RT^{k}$ and $v_h = i_h \varphi \in V_{0,D}^k$
with $\Gamma_N = \emptyset$, then 
applying the
Cauchy-Schwarz inequality,
%the approximation properties of the Lagrange and Raviart-Thomas interpolants 
\cref{eq:approx_Lagrange} and \cref{eq:approx_RT} that
 \begin{equation} \label{L2-3}
 \begin{split}
&(z_h, \nabla \cdot R_h \bfq + \mu i_h \phi) = -\left(\!\bfp_h \!-\! \bfbeta u_h + A \nabla u_h, 
R_h \bfq - \bfbeta i_h \phi + A \nabla (i_h \phi) \right)\\
\lesssim& \|\bfp_h - \bfbeta u_h + A \nabla u_h\|
\left(\|A\|_{\infty}\|\nabla ( \phi - i_h \phi )\| \!+\!
\|\bfbeta\|_{\infty} \| \|   \phi - i_h \phi\|
\!+\!\|\bfq - R_h  \bfq \| \right) \\
\lesssim& 
\|(\bfp_h - \bfbeta u_h + A \nabla u_h)\|
\left( \|A\|_{\infty} h + \|\bfbeta\|_{\infty} h^2 + h\right) \|\phi\|_{H^2(\Omega)} \\
\lesssim& h\|(\bfp_h - \bfbeta u_h + A \nabla u_h)\| \|z_h\|.
\end{split}
\end{equation}
\cref{L2-multiplier} is then a direct consequence of \cref{L2-1}, \cref{L2-3} and \cref{Residual-estimate}.
%Then combining (\cref{L2-1})--(\cref{L2-3}) 
%implies
%\begin{eqnarray*}
%\|z_h\|_\Omega 
%&\lesssim & h  \\
%&\lesssim&
%C_u h^{k+1}|u|_{H^{k+1}}
%+ h^{k+1} (|\bfp|_{H^{k}(\Omega)}+|\nabla \cdot \bfp|_{H^{k-1}(\Omega)}),
%\end{eqnarray*}
This completes the proof of the proposition.
\end{proof}

We now proceed to prove the error estimation of the primal variable in the  $L^2$ norm.
To estimate the error of the primal variable in the $L^2$ norm we require that the adjoint problem is well-posed and
satisfies a shift theorem for the $H^2$ semi-norm. 
\begin{assumption} \label{Assumption_2}
Consider the adjoint problem for \cref{CD_strong}. For each $\zeta \in L^2(\Omega)$, 
we assume that the data are such that the following adjoint problem admits an unique solution, using Fredholm's alternative,
\begin{equation}\label{eq:primal_wp}
- \nabla \cdot A \nabla \varphi - \bfbeta \cdot \nabla \varphi + \mu \varphi = \zeta \mbox{ in } \Omega
\end{equation}
%\HOX{ADAPT THIS TO THE ADVECTION DIFFUSION CASE}
with $\varphi\vert_{\partial \Omega} = 0$.
Furthermore, the following regularity result holds true:
\begin{equation}\label{eq:adjoint_stab}
\|\varphi\|_{H^2(\Omega)} \lesssim \|\zeta\|.
\end{equation}
\end{assumption}

\begin{prop}\label{prop:well_posed_L2}
Let $u \in H^{k+1}(\Omega) \cap H^1_0(\Omega)$, $ \bfp \in H^{k}(\Omega)^d$
and $(u_h, \bfp_h, z_h)$ be the
solution of \cref{eq:EL_1}--\cref{eq:EL_2}. Under the 
\cref{Assumption_2} we have 
\begin{equation} \label{L2_estimate}
\|u - u_h\|\lesssim h^{k+1} \left( |u|_{H^{k+1}(\Omega)}+  |\bfp|_{H^{k}(\Omega)} \right).
\end{equation}
\end{prop}
\begin{proof}
Let $\varphi$ be the solution of the dual problem \cref{eq:primal_wp} with right hand side being $e:= u -u_h$. 
Then by integration by parts, the assumption that $\varphi = 0 $ on $\partial \Omega$, we have
\begin{equation}\label{L2-two terms}
\begin{split}
	\|e\|^2 
	&=(f, \varphi) +(u_h,  \nabla \cdot A \nabla \varphi +\bfbeta \cdot \nabla \varphi - \mu \varphi)
	\\
	&=
	(f - \nabla \cdot \bfp_h - \mu u_h, \varphi)
	- ( \bfp_h + A \nabla u_h - \bfbeta u_h, \nabla \varphi).
\end{split}
\end{equation}
The first term can be estimated by applying \cref{eq:EL_2} and the Cauchy-Schwartz inequality:
\begin{equation}\label{L2-sub term-a}
	(f - \nabla \cdot \bfp_h - \mu u_h, \varphi) =(f - \nabla \cdot \bfp_h - \mu u_h, \varphi - \pi_{X,k}
	 \varphi)
	 \lesssim h \tnorm{(u-u_h,\bfp - \bfp_h)}_{-1} \|\varphi\|_{H^2(\Omega)}.
\end{equation}
To estimate the second term we apply \cref{eq:EL_1} with $\bfq_h =
R_h (\nabla \varphi) \in RT^{k}$ and the fact that
$ \nabla \cdot (R_{h} \nabla \varphi))  =  \pi_{X,k} \triangle \varphi$:
\begin{equation}\label{L2-sub term-b}
\begin{split}
	&( \bfp_h + A \nabla u_h - \bfbeta u_h, \nabla \varphi ) \\
	=&
	-( z_h,  \nabla \cdot (R_{h} \nabla \varphi)) +
	 ( \bfp_h + A \nabla u_h - \bfbeta u_h, (\nabla \varphi - R_{h} \nabla \varphi))\\
	 \lesssim&\,
	 \|z_h\| \|\pi_{X,k} \triangle \varphi\| + 
	 h \| \bfp_h + A \nabla u_h - \bfbeta u_h\|\| \varphi \|_{H^2(\Omega)}\\
	 \lesssim&\,
	 (\|z_h\| + h\| \bfp_h + A \nabla u_h - \bfbeta u_h\| ) \|\varphi\|_{H^2(\Omega)}.
	 \end{split}
\end{equation}
Combing \cref{L2-two terms}--\cref{L2-sub term-b}  gives
\begin{equation}\label{L2-last}
	\|e\| \lesssim 
	 h\tnorm{(u-u_h,\bfp - \bfp_h)}_{-1}+
	\|z_h\|. 
\end{equation}
      \cref{L2_estimate} is then a direct consequence of \cref{L2-last}, \cref{Residual-estimate} and \cref{L2-multiplier}.
This completes the proof of the proposition.
\end{proof}

\begin{rem}
Observe that the hidden constants in \cref{L2-stability} and
\cref{eq:adjoint_stab} typically blow up in the advection dominated
regime. Therefore the above $L^2$-analysis is relevant only in the
low Peclet regime.
\end{rem}
\subsection{Error estimates in the advection dominated regime}
In this section, we analyze the error estimates in the advection dominated regime.
For the stability we make the following assumption on the data that ensures stability of the
adjoint equation independent of the diffusivity, see \cite{da1986stationary}.
\begin{assumption}\label{assumption:advection-regime}
We assume that the the domain $\Omega$ is convex, that the diffusivity $A$ is a scalar
 and $\bfbeta_\infty=O(1)$. We also introduce the following condition on
data. Let 
$\mathcal{I}$ denote the identity matrix and
$\nabla_S \bfbeta :=1/2( \nabla \bfbeta +
(\nabla \bfbeta)^T)$, i.e., the symmetric part of $\nabla
\bfbeta$. Then assume that
$\mu \mathcal{I}- (\nabla_S \bfbeta - \dfrac{1}{2} \nabla \cdot \bfbeta \mathcal{I} )$ is symmetric positive definite and denote by $\Lambda_{min}$ its smallest eigenvalue. Moreover we assume that
$\bfbeta \cdot \bn = 0$ on $\partial \Omega$.
\end{assumption}

We first prove the following inverse inequality regarding the $H^{-1}(\Omega)$ norm.
\begin{lem}
For any $v_h \in V_h^k$ the following inverse inequality holds: 
\begin{equation}\label{negative-norm-inverse-inequality}
\|v_h\|_{L^2(\Omega)} \lesssim h^{-1} \|v_h\|_{H^{-1}(\Omega)}
\end{equation}
\end{lem}
\begin{proof}
	Let $E \in H_0^1(\Omega)$ be the weak solution to
	\[
		- \triangle E +E = v_h \mbox{ in } \Omega.
	\]
	Then by the definition and duality inequality we have
	\begin{equation}\label{inverse-equivalence}
	\begin{split}
		\|E\|_{H^1(\Omega)} &= 
		\sup_{\substack{w \in H_0^1(\Omega) \\ \|w\|_{H^1(\Omega}=1}}
		\left((\nabla E, \nabla w)+ (E,w) \right)\\
		&=
		\sup_{\substack{w \in H_0^1(\Omega) \\ \|w\|_{H^1(\Omega)}=1}} (-\triangle E + E, w)
		\le \|v_h\|_{H^{-1}(\Omega)}.
	\end{split}
	\end{equation}
	By integration by parts we also have
	\begin{eqnarray*}
		\| v_h\|^2 = (v_h, -\triangle E +E) = (v_h, E) + (\nabla v_h, \nabla E) 
		\le \|v_h\|_{H^1(\Omega)} \|E\|_{H^1(\Omega)},
	\end{eqnarray*}
	which, combining with \cref{inverse-equivalence} and the inverse inequality, gives \cref{negative-norm-inverse-inequality}.
This completes the proof of the lemma.	
\end{proof}

\begin{lem}\label{lem:stability}
Let $\phi \in H_0^1(\Omega)$ be the solution to \cref{eq:primal_wp} with the right side being 
$\psi \in H_{0}^1(\Omega)$. 
 Then under the \cref{assumption:advection-regime} the following stability result holds:
\begin{equation}
	\Lambda_{min}\| \nabla \phi\| \le \| \nabla \psi\|.
\end{equation}
\end{lem}
\begin{proof}
By the definition and integration by parts we have
	\begin{eqnarray*}
		(\psi, - \triangle \phi) = 
		(\mu \nabla \phi, \nabla \phi) 
		+(\bfbeta \cdot \nabla \phi, \triangle \phi) + (A \triangle \phi, \triangle \phi)
	\end{eqnarray*}
Using the relation  of \cite[equation (3.6)]{burman2014robust} we
have for the second term of the right hand side
\begin{equation}
	(\bfbeta \cdot \nabla \phi, \triangle \phi) = 
	\left( \left(\dfrac{1}{2} \nabla \cdot \bfbeta \mathcal{I} -
	 \nabla_S \bfbeta \right) \nabla \phi, \nabla \phi \right).
\end{equation}

Combining similar terms we then have
\begin{equation}
	\left(
	\left(\mu \mathcal{I}- (\nabla_S \bfbeta - \dfrac{1}{2} \nabla \cdot \bfbeta \mathcal{I} ) \right)
	\nabla \phi, \nabla \phi
	\right)
	+ (A \triangle \phi, \triangle \phi)=
	(\psi, - \triangle \phi) = (\nabla \psi, \nabla \phi),
\end{equation}
and, therefore,
\begin{equation}
	\Lambda_{min} \| \nabla \phi\| \le \|\nabla \psi\|
\end{equation}
and, as a by product,
\[
	\|A^{1/2} D^2 \phi\| \lesssim \| A^{1/2} \triangle \phi\| \le \Lambda_{min}^{-1/2} \| \nabla \psi\|.
\]
This completes the proof of the lemma.
\end{proof}

\begin{prop}\label{prop:neg_norm}
Let $u$ and $u_h$ be the solution of \cref{weak_CD} and \cref{eq:compact}, respectively.
Then under the \cref{assumption:advection-regime} we have the following estimate:
	\begin{equation}\label{inverse-norm-estimate}
		\|u-u_h\|_{H^{-1}(\Omega)} \leq C_P \Lambda_{min}^{-1} \tnorm{u-u_h, \bfp- \bfp_h}
	\end{equation}
where $C_P$ is the constant of the Poincar\'e inequality $$\sum_K \|h_{K}^{-1}(\phi - \pi_{X,0} \phi )\|^2 \leq
C_P^2 \|\nabla \phi\|^2$$ and $\Lambda_{min}$ is defined in
\cref{assumption:advection-regime}. $C_P = \pi^{-1}$ on convex
domains, see \cite{Beb13}.
\end{prop}
\begin{proof}
	By definition we have 
	\[
		\|u-u_h\|_{H^{-1}(\Omega)} = 
		\sup_{\substack{w \in H_0^1(\Omega) \\ \|w\|_{H^1(\Omega)}=1}}
		(e_u, w).
	\]
	Let $\varphi \in H_0^1(\Omega)$ be the solution of
        \cref{eq:primal_wp} with the right hand side an arbitrary
        function $\psi \in
        H^1_0(\Omega)$ with $\|\psi\|_{H^1(\Omega)} = 1$.
	Applying the integration by parts, \cref{eq:EL_2} and the Cauchy-Schwartz inequality gives
	\begin{eqnarray*}
		(e_u, \psi) &=& (e_u, \mu \varphi - \bfbeta \cdot \nabla \varphi - A \triangle \varphi)
		\\
		&=& (\mu e_u + \nabla \cdot \bfe_p, \varphi)
		+(\bfe_p + A \nabla e_u - \bfbeta e_u, \nabla \varphi)\\
		&=&
		(\mu e_u + \nabla \cdot \bfe_p, \varphi - \pi_{X,0} \varphi)
		+(\bfe_p + A \nabla e_u - \bfbeta e_u, \nabla \varphi)\\
		&\le&
		\left( C_P  \|h(\mu e_u + \nabla \cdot \bfe_p )\|  +
		\|\bfe_p + A \nabla e_u - \bfbeta e_u\|
		\right) \| \nabla \varphi \|
		\\
		& \le & C_P\tnorm{(u-u_h, \bfp - \bfp_h)}_{-1}\| \nabla \varphi
                        \| \leq C_P \Lambda_{min}^{-1}\tnorm{(u-u_h, \bfp - \bfp_h)}_{-1}.
	\end{eqnarray*}
	where in the last inequality we also applied the stability result of \cref{lem:stability}.
	This completes the proof of the proposition, since the bound
        is valid for arbitrary $\psi \in
        H^1_0(\Omega)$ with $\|\psi\|_{H^1(\Omega)} = 1$.
\end{proof}
\begin{cor}[Negative norm, a posteriori and a priori bounds]
Under the same hypothesis as Proposition \ref{prop:neg_norm} the
following a posteriori and a priori error estimates hold:
	\begin{multline}\label{inverse-norm-apost}
		\|u-u_h\|_{H^{-1}(\Omega)} \leq C_P \Lambda_{min}^{-1}
                (\|h(f - \mu u_h - \nabla \cdot \bfp_h)\| + \|\bfbeta
                u_h - A\nabla u_h - \bfp_h\|) \\
 \lesssim  C_P
                \Lambda_{min}^{-1}(h^{k+1} |u|_{H^{k+1}(\Omega)}+ h^{l+1} |\bfp|_{H^{l+1}(\Omega)})
	\end{multline}
\end{cor}
\begin{proof}
The proof is immediate using Proposition \ref{prop:neg_norm} and
Corollary \ref{col:residual-estimate-for-advection}.
\end{proof}
We are now ready to prove the main result.
\begin{thm}\label{lem:L2-estimate-for-u}
	Let $u \in H^{k+1}(\Omega) \cap H^1_0(\Omega)$, $ \bfp \in H^{k+1}(\Omega)^d$
and $(u_h, \bfp_h, z_h)$ be the
solution of \cref{eq:EL_1}--\cref{eq:EL_2}. Assume that $A \lesssim h^2$.
Then under the \cref{assumption:advection-regime} we have the following error estimate:
\begin{equation}\label{error-estimate-for-advection}
	\| u - u_h\| + \| \nabla \cdot (\bfp - \bfp_h)\| \lesssim h^{k} \left( |u|_{H^{k+1}(\Omega)} +  |\bfp|_{H^{k+1}(\Omega)}  \right).
\end{equation}
Furthermore, if $\divvr \bfp \in H^{k+1}(\Omega)$ we have
\begin{equation}\label{error-estimate-for-streamline}
	\|\nabla\cdot \bfbeta (u - u_h)\| 
	\lesssim h^{k} \left( |u|_{H^{k+1}(\Omega)} +  |\bfp|_{H^{k+1}(\Omega)} + 
	|\divvr \bfp |_{H^{k+1}(\Omega)}  \right).
\end{equation}
Here the hidden constants are bounded in the limit as $A \rightarrow 0$.
\end{thm}
\begin{proof}
Applying the triangle inequality, \cref{negative-norm-inverse-inequality}, \cref{inverse-norm-estimate}
and \cref{col:residual-estimate-for-advection} gives
\begin{equation}\label{L2-estimate-for-u}
\begin{split}
	\| u -u_h\| &\le \| u - i_h u\| + h^{-1}\| u_h - i_h u\|_{H^{-1}(\Omega)}\\
	&\lesssim h^{k+1} |u|_{H^{k+1}(\Omega)} + h^{-1}\|u-u_h\|_{H^{-1}(\Omega)} 
	+
	h^{-1}\|u-i_h u\|_{H^{-1}(\Omega)} \\[2mm]
	&\lesssim h^{k+1} |u|_{H^{k+1}(\Omega)} + h^{-1}\|u-u_h\|_{H^{-1}(\Omega)} 
	+
	h^{-1}\|u-i_h u\|_{L^2(\Omega)} \\[2mm]
	&\lesssim
	h^{k} |u|_{H^{k+1}(\Omega)} + h^{-1} \tnorm{(u-u_h, \bfp - \bfp_h)}\\[2mm]
	&\lesssim
	h^{k} \left( |u|_{H^{k+1}(\Omega)} +  |\bfp|_{H^{k+1}(\Omega)}  \right).
\end{split}
\end{equation}
Applying the triangle inequality,  \cref{L2-estimate-for-u} and \cref{Residual-estimate} gives
\begin{equation}\label{div-estimate-for-p}
\begin{split}
\| \nabla \cdot (\bfp - \bfp_h)\| &\le \|\nabla \cdot (\bfp - \bfp_h) + \mu (u- u_h)\| + \|\mu(u-u_h)\|\\
&\lesssim h^{-1} \tnorm{(u-u_h, \bfp - \bfp_h)} + \|\mu(u-u_h)\|\\
&\lesssim
	h^{k} \left( |u|_{H^{k+1}(\Omega)} +  |\bfp|_{H^{k+1}(\Omega)}  \right).
\end{split}
\end{equation}
\cref{error-estimate-for-advection} is then a direct result of \cref{L2-estimate-for-u} and \cref{div-estimate-for-p}.

To prove \cref{error-estimate-for-streamline} we first apply the triangle inequality,
\begin{equation}\label{streamline-estimate}
\|\nabla\cdot (\bfbeta (u - u_h))\| \leq \|\nabla \cdot (\bfbeta (u - i_h
u))\| +\|\nabla \cdot (\bfbeta (u_h - i_h u))\|.
\end{equation}
The first term in \cref{streamline-estimate} can be easily estimated
using \cref{eq:approx_Lagrange}. For the second term in
\cref{streamline-estimate} applying the triangle and inverse inequalities gives
\begin{equation}\label{streamline-estimate-a}
\begin{split}
\|\nabla \cdot (\bfbeta (u_h - i_h u))\|  
&\leq
h^{-1} \|\bfbeta (u_h - i_h u) - (R_h \bfp - \bfp_h) - A \nabla (u_h - i_h u) \| \\
&\quad +
	 h^{-2} \|A\|_{\infty} \left( \| (u_h - u) \|  +\| (i_h u - u) \|\right) \\
&\quad + \|\nabla \cdot (\bfp - \bfp_h)\| + \|\nabla \cdot (\bfp - R_h \bfp)\|.
\end{split}
\end{equation}
By the triangle inequality, \cref{col:residual-estimate-for-advection} and \cref{eq:tnorm_approx}
we have
\begin{equation}%\label{07170}
	 h^{-1}\|\bfbeta (u_h - i_h u) \!-\!(R_h \bfp - \bfp_h) \!-\! A \nabla (u_h - i_h u) \| \lesssim 
	 h^k( |u|_{H^{k+1}(\Omega)} + |\bfp|_{H^{k+1}(\Omega)}).
\end{equation}
By the smallness assumption $\|A\|_{\infty} \lesssim  h^2$,
\cref{eq:approx_Lagrange}, \cref{eq:approx_RT},
\cref{L2-estimate-for-u} and \cref{div-estimate-for-p} the remaining terms in \cref{streamline-estimate-a} can be estimated as follows:
\begin{equation}\label{streamline-estimate-b}
\begin{split}
&h^{-2} \|A\|_{\infty} \left( \| (u_h - u) \|  +\| (i_h u - u) \|\right) 
 + \|\nabla \cdot (\bfp - \bfp_h)\| + \|\nabla \cdot (\bfp - R_h \bfp)\|\\
 \lesssim& \, 
  h^{k} \left( |u|_{H^{k+1}(\Omega)} +  |\bfp|_{H^{k+1}(\Omega)} + 
	|\divvr \bfp |_{H^{k+1}(\Omega)}  \right).
\end{split}
\end{equation}
Finally, \cref{error-estimate-for-streamline} is a direct consequence of \cref{streamline-estimate}--\cref{streamline-estimate-b}.
This completes the proof of the lemma.
\end{proof}
\begin{remark}
It is possible to prove Theorem \ref{lem:L2-estimate-for-u} under the
standard coercivity assumption for advection-diffusion problems, but
not Proposition \ref{prop:neg_norm}. Also note that we need the diffusivity to be
$O(h^2)$ to ensure that the high Peclet result holds. This is a
stronger assumption than is usual for convection--diffusion equations.
\end{remark}

\section{Numerical Experiments}\label{sec:numerical-results}
In this section we present results for numerical experiments in both the diffusion dominated and convection
dominated regimes. The numerical results are produced using the FEniCS software \cite{logg2012automated}.
\begin{example}[Boundary Layer]\label{ex:1}
In this example we consider the boundary layer problem \cite{mitchell2013}
\[
	- \epsilon \triangle u + 2 \dfrac{\partial u}{\partial x} + \dfrac{\partial u}{\partial y} =f
\]
on the domain $\Omega = (-1,1) \times (-1,1)$ where the true solution 
has the following representation:
\[
	u = (1 - \exp(-(1-x)/\epsilon))*(1 - \exp(-(1-y)/\epsilon))*\cos(\pi(x+y))
\]
and $\epsilon \in \mathbb{R}$.
The solution has a $O(\epsilon)$ boundary layer along the right top sides of the domain and the
value of $\epsilon$ determines the strength of the boundary layer.
\end{example}

We first test the value $\epsilon =1$ in which case the solution is smooth and we aim to test the optimal convergence rates of our method for smooth problems.
The magnitude of the errors and their corresponding convergence rates are listed in 
\cref{tab:ex1-a} for the first and second orders, i.e., $k=1$ and $k=2$.
 For both orders we observe the optimal convergence performance 
 for the primal variable in both the $L^2$ and $H^1$ norms. For the flux variable 
 we are able to observe the optimal rate for the linear order, and, however, a
 slightly suboptimal rate for the second order. Nevertheless the flux
 variable provides an approximation of the flux that is more accurate
 than that using the primal variable by two orders of magnitude.
 
 \begin{table}[h!]
\caption{ \footnotesize{Errors and convergence performance for \cref{ex:1} with $\epsilon=1$}}
\label{tab:ex1-a}
\begin{center}
\subfloat[order =1]{
%\resizebox{\textwidth}{!}{
	\begin{tabular}{||c|c c|c c| c c||}
	\hline
 	%\multicolumn{9}{||c||}{$\epsilon$ = 1 \hspace{2mm} order =1} \\
 	%\hline
	h & $\|u - u_h\|$ & rate & $\|u-u_h\|_{H^1(\Omega)}$ & rate & $\| \bfp -\bfp_h\|$ & rate 
	\\
	\hline
%	1/8  & 3.090E-2 &      & 6.503E-1&      & 5.332E-2 &\\
%	1/16 & 7.865E-3 & 1.97 & 3.282E-1 & 0.99 &1.338E-2 & 1.99\\
	1/32 & 1.975E-3 & 1.99 & 1.644E-1 & 1.00 &3.346E-3 & 2.00 \\
	1/64 & 4.942E-4 & 2.00 & 8.229E-2& 1.00  &8.367E-4 & 2.00\\
	1/128& 1.235E-4 & 2.00 & 4.115E-2 & 1.00 &2.091E-4 & 2.00 \\
	\hline
	\end{tabular}
}
\end{center}

\begin{center}
\subfloat[order =2]{
%	\resizebox{\textwidth}{!}{
	\begin{tabular}{||c|c c|c c| c c||}
	\hline
 	%\multicolumn{9}{||c||}{$\epsilon$ = 1 \hspace{2mm} order =2} \\
 	%\hline
	h & $\|u - u_h\|$ & rate & $\|u-u_h\|_{H^1(\Omega)}$ & rate & $\|\bfp-\bfp_h\|$&rate \\
	\hline
%	1/8  & 1.011E-3 &      & 6.098E-2 &      & 2.997E-3 &\\
%	1/16 & 1.279E-4 & 2.98 & 1.539E-2 & 1.99 & 4.587E-4 & 2.71\\
	1/32 & 1.605E-5 & 2.99 & 3.857E-3 & 2.00 & 7.368E-5 & 2.64\\
	1/64 & 2.008E-6 & 3.00 & 9.650E-4 & 2.00 & 1.230E-5 & 2.60 \\
	1/128& 2.510E-7 & 3.00 & 2.412E-4 & 2.00 & 2.109E-6 & 2.54\\
	\hline
	\end{tabular}
	}
		\end{center}
\end{table}

%\begin{table}[ht!]
%\caption{ \footnotesize{errors and convergence rates for Example 1 with $\epsilon=1$ and $k=2$}}
%\label{table_12}
%\begin{center}
%	\resizebox{\textwidth}{!}{
%	\begin{tabular}{||c|c c|c c| c c|c c||}
%	\hline
% 	\multicolumn{9}{||c||}{$\epsilon$ = 1 \hspace{2mm} order =2} \\
% 	\hline
%	h & $\|u - u_h\|$ & rate & $\|u-u_h\|_{H^1(\Omega)}$ & rate & $\|p-p_h\|$ &rate& $\|z_h\|$&rate\\
%	\hline
%%	1/8  & 1.011E-3 &      & 6.098E-2 &      & 2.997E-3 &\\
%%	1/16 & 1.279E-4 & 2.98 & 1.539E-2 & 1.99 & 4.587E-4 & 2.71\\
%	1/32 & 1.605E-5 & 2.99 & 3.857E-3 & 2.00 & 7.368E-5 & 2.64&6.942E-3&2.99\\
%	1/64 & 2.008E-6 & 3.00 & 9.650E-4 & 2.00 & 1.230E-5 & 2.60&8.690E-4&3.00 \\
%	1/128& 2.510E-7 & 3.00 & 2.412E-4 & 2.00 & 2.109E-6 & 2.54&1.086E-4&3.00\\
%	\hline
%	\end{tabular}
%	}
%		\end{center}
%		\end{table}
We then test the method in the advection dominated regime with boundary layer by letting $\epsilon = 0.01$ (see performance results in \cref{tab:ex1-b}). 
\textcolor{black}{For both the first and second orders, the method produce the optimal convergence rate for the 
streamline derivative when the layer is resolved. For the flux variable we observe the optimal convergence for both orders 1 and 2 (with even super convergence result for order 2).
For the primal variable we observe the optimal convergence rates both
in the $L^2$- and $H^1$- norms (with super convergence for the $L^2$-
norm in the second order case).}
\begin{table}[h!]
\caption{ \footnotesize{Errors and convergence rates for \cref{ex:1} with $\epsilon=0.01$}}
\label{tab:ex1-b}
\begin{center}
\subfloat[order = 1]{
	\begin{tabular}{||c|c c|c c| c c||}
	\hline
% 	\multicolumn{7}{||c||}{$\epsilon$ = 0.01 \hspace{2mm} order =1} \\
% 	\hline
	h & $\|u - u_h\|$ & rate & $\|u-u\|_{H^1(\Omega)}$ & rate & $\|\bfp -\bfp_h\|$ & rate\\
	\hline
%	1/8  & 4.007E-1 &      & 5.781 &      & 8.884E-1 &\\
%	1/16 & 3.592E-1 & 0.16 & 6.645 & -0.2& 7.966E-1 & 0.16\\
	1/32 & 1.979E-1 & 0.86 & 6.073 & 0.13 & 4.394E-1 & 0.86\\
	1/64 & 7.195E-2 & 1.46 & 4.009  & 0.60 & 1.599E-1 & 1.46 \\
	1/128& 2.044E-2 & 1.82 & 2.189 & 0.87 & 4.547E-2 & 1.81\\
	1/256& 5.296E-3 & 1.95 & 1.120   & 0.97& 1.178E-2 & 1.95\\
	\hline
	 \hline
	h & $\| \divvr(\bfp - \bfp_h)\|$ & rate & 
	    $\|\divvr ( \bfbeta(u - u_h))\|$ & rate &
	    $\|z_h\|$&rate\\
	\hline
%	1/8  & 4.040E-0 &      & 8.315 & &1.088E-2&     \\
%	1/16 & 2.628E-0 & 0.62 & 10.048 & -0.27&6.369E-3&0.77 \\
	1/32 & 1.277E-0 & 1.04 & 9.462& 0.09&2.592E-3&1.30 \\
	1/64 & 4.354E-1 & 1.55 & 6.314& 0.58&7.917E-4&1.71  \\
	1/128& 1.206E-1 & 1.85 & 3.452& 0.87&2.092E-4&1.92 \\
	1/256& 3.100E-2 & 1.96 & 1.766& 0.97&5.300E-5&1.98 \\
	\hline
	\end{tabular}
}
\end{center}

\begin{center}
\subfloat[order=2]{
	\begin{tabular}{||c|c c|c c| c c||}
%	\hline
% 	\multicolumn{7}{||c||}{$\epsilon$ = 0.01 \hspace{2mm} order =2} \\
 	\hline
	h & $\|u - u_h\|$ & rate & $\|u-u\|_{H^1(\Omega)}$ & rate & $\|\bfp - \bfp_h\|$ &rate\\
	\hline
%	1/8  & 2.098E-1 &      & 4.378E-0 &      & 4.484E-1 &\\
%	1/16 & 1.277E-1 & 0.72 & 4.061E-1 & 0.11 & 2.866E-1 & 0.76\\
	1/32 & 3.461E-2 & 1.88 & 2.228E-1 & 0.87 & 7.666E-2 & 1.90\\
	1/64 & 4.772E-3 & 2.86 & 7.838E-1 & 1.51 & 1.036E-2 & 2.89 \\
	1/128& 4.483E-4 & 3.41 & 2.194E-1 & 1.84 & 9.210E-4 & 3.49\\
	1/256& 4.186E-5 & 3.42 & 5.665E-2 & 1.95 & 7.772E-5 & 3.57\\
	\hline
	 \hline
	h & $\| \divvr(\bfp - \bfp_h)\|$ & rate & 
	    $\|\divvr ( \bfbeta(u - u_h))\|$ & rate &
	    $\|z_h\|$&rate\\
	\hline
%	1/8  & 2.151E-0 &      & 6.355E-0 &      &3.352E-3&     \\
%	1/16 & 1.081E-0 & 0.99 & 6.272E-0 & 0.02 &1.540E-3&1.21 \\
	1/32 & 3.138E-1 & 1.79 & 3.505E-0 & 0.84 &3.589E-4&2.10 \\
	1/64 & 5.672E-2 & 2.47 & 1.235E-0 & 1.50 &4.582E-5&2.97  \\
	1/128& 7.992E-3 & 2.83 & 3.456E-1 & 1.84 &3.880E-6&3.56 \\
	1/256& 1.032E-3 & 2.95 & 8.919E-2 & 1.95 &2.935E-7&3.72 \\
	\hline
	\end{tabular}
	}
\end{center}
\end{table}

\textcolor{black}{To test the robustness of our method, \cref{fig:osillation} shows the numerical solutions on structured meshes of various element sizes using the first order method for the problem with $\epsilon=0.002$ in which case the boundary layer is extremely sharp. More precisely the mesh sizes are chosen such that the boundary layer are under resolved $(h = 1/64)$, half resolved $(h=1/512)$ and fully resolved ($h=1/1024$).  
%We observe that despite the existence of extreme sharp boundary layer and solving on relatively coarse meshes, the numerical solutions, however, have no signs of global spurious oscillations that damages the approximation. 
We observe that when the mesh size could not resolve the boundary layer fully
global oscillations appear in the approximation solution and
jeopardize the quality of the approximations. Observe that the
oscillations here are different to those appearing in the standard
Galerkin method.} In \cref{sec:outflow}
we propose two simple strategies to tackle this issue. 

\begin{figure}[h!]
    \centering
    \includegraphics[width=0.6\textwidth]{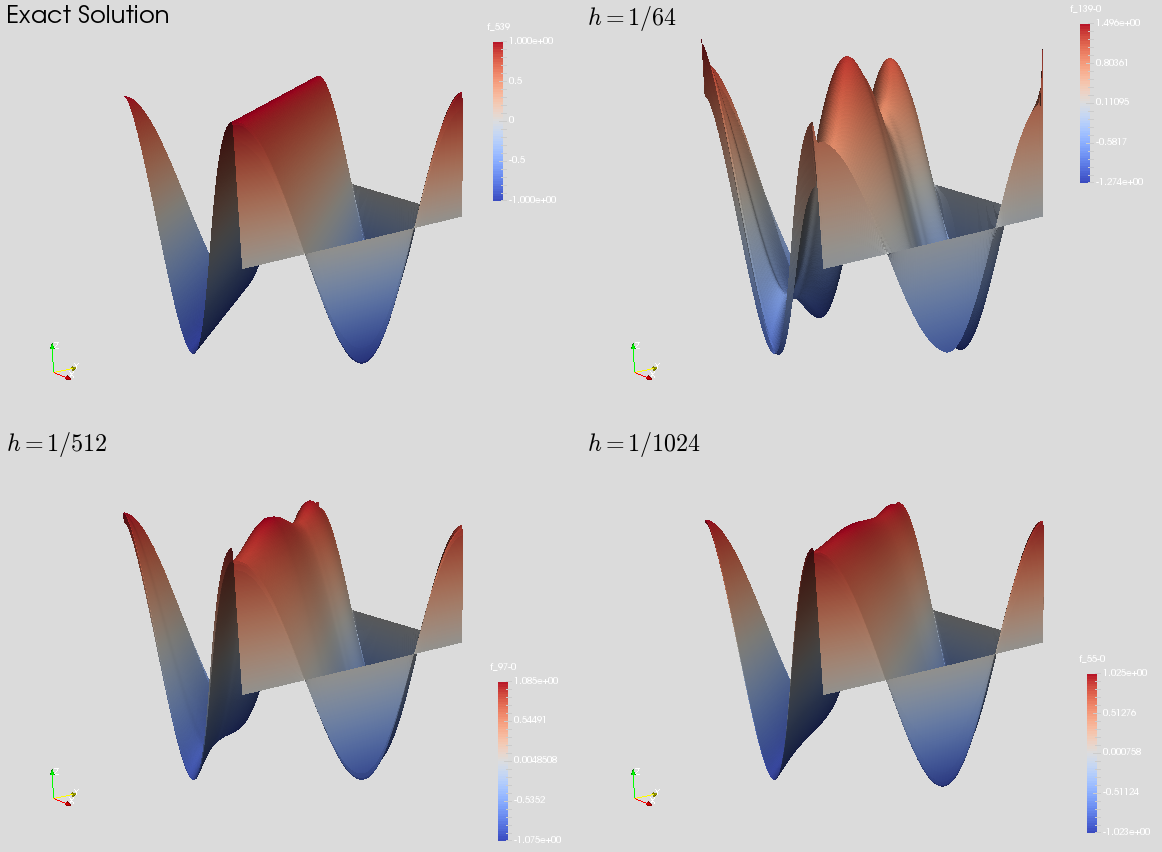}
    \caption{ Various numerical solutions for \cref{ex:1} with  $\epsilon=0.002$}
    \label{fig:osillation}
\end{figure}

\begin{example}[Reentrant Corner]\label{ex:2}
In this example we test a pure diffusion problem, i.e.,
 $\epsilon = 1$, $\bfbeta=0$ and $\mu=0$, on the L-shaped doman $\Omega = (-1,1)^2 \setminus (-1,-1)^2$.
We consider the problem with solution being
\[
	u(r,\theta) = r^{2/3}\sin(2\theta/3), \quad \theta \in [0, 3\pi/2],
\]
in polar coordinates.
 It is well 
known that the solution satisfies 
\[-\triangle u =0 \quad \mbox{in} \, \Omega\]
 and belongs to $H^{5/3 -\epsilon}(\Omega)$ for $\epsilon>0$ with the singularity located at the reentrant corner, i.e., $(0,0)$. The numerical scheme takes the pure Dirichlet boundary condition.
 \end{example}

\begin{table}[ht!]
\caption{ \footnotesize{Errors and convergence rates for \cref{ex:2} with order $=1$ }}
\label{tab:ex2}
\begin{center}
	\begin{tabular}{||c|c c|c c| cc ||}
	\hline
	h & $\|u - u_h\|$ & rate & $\|u-u_h\|_{H^1(\Omega)}$ & rate& $\|\bfp-\bfp_h\|$ & rate\\
	\hline
%	1/4  & 1.665E-2 &      &1.923E-1&    &1.034E-1& \\
%	1/8  & 7.726E-3 & 1.10 &1.220E-1&0.66&6.524E-2&0.66\\
	1/16 & 3.025E-3 & 1.35 &7.790E-2&0.65&4.110E-2&0.67\\
	1/32 & 1.189E-3 & 1.35 &4.949E-2&0.65&2.589E-2&0.67\\
	1/64 & 4.689E-4 & 1.34 &3.315E-2&0.66&1.631E-2&0.67\\
%	1/128& 1.853E-4 & 1.34 &1.982E-2&0.66&1.027E-2&0.67\\
%	1/256& 7.335E-5 & 1.34 &1.251E-2&0.66&6.474E-3&0.67\\
%	1/512& 2.906E-5 & 1.34 &7.896E-3&0.66&4.078E-3&0.67\\
	\hline
	\end{tabular}
\end{center}
\end{table}

The magnitude of errors and their corresponding convergence rates are presented in 
\cref{tab:ex2}. For this pure diffusion problem, where the solution
has singularity and with limited smoothness, we observe the optimal
convergence performance for both the primal and flux variables. The
flux variable is still a superior approximation of the fluxes, but
here only by a factor two.

\begin{example}[Internal Layer]\label{ex:3}
In this example we consider a pure advection problem
\cite[Section 5.2.3]{EG04}.
The solution has the following representation:
\begin{equation}\label{ex3}
u(x,y) = exp( - \sigma \rho(x,y)\mbox{arccos}\left(\dfrac{y+1}{\rho(x,y)}\right)\mbox{arctan}\left(\dfrac{\rho(x,y)-1.5}{\delta}\right)
\end{equation}
where $\sigma =0.1$, $\rho(x,y) = \sqrt{x^2+(y+1)^2}$. It is easy to verify that
\[
	\divvr \bfbeta =0, \quad \bfbeta \cdot \nabla u + \sigma u = 0 
\]
for $\bfbeta = \dfrac{1}{\rho(x,y)}(y+1,x)$,
and that the inflow boundary,
$\Gamma^- = \{ x \in \partial \Omega, \bfbeta(x) \cdot \bn < 0\}$, is 
$x = 0$ and $y=1$.
The finite element scheme we use for this problem is to find 
$u \in V_{g, \Gamma^{-}}^k$, $\bfp \in RT^k$, and $z_h \in X_h^k$ such that 
\cref{eq:EL_1} and \cref{eq:EL_2} hold.
\end{example}

We first test the case when $\delta =1$ to test the performance of our method on smooth problems
(see performance results in \cref{tab:ex3-a}). We further test the case for $\delta = 0.01$ in which case the solution has a sharp internal layer (see performance results in \cref{tab:ex3-b}).

\begin{table}[ht!] 
\caption{\footnotesize{Errors and convergence rates for \cref{ex:3} with $\delta = 1$ and order $=1$}}
\label{tab:ex3-a}
\begin{center}
	\begin{tabular}{||c|c c|c c| c c||}
	\hline
%	\multicolumn{7}{||c||}{$\delta$ = 1 \hspace{2mm} order =1} \\
% 	\hline
	h & $\|u - u_h\|$ & rate & $\|u-u\|_{H^1(\Omega)}$ & rate & 
	$\|\bfp-\bfp_h\|$ & rate\\
	\hline
%	1/8  & 9.885E-4 &      & 2.782E-2 &      & 1.146E-3&\\
%	1/16 & 2.475E-4 & 2.00 & 1.723E-2 & 0.69 & 2.828E-4 & 2.02\\
	1/32 & 6.021E-5 & 2.04 & 8.591E-3 & 1.00 & 6.939E-5 & 2.03\\
	1/64 & 1.475E-5 & 2.03 & 4.281E-3 & 1.00 & 1.711E-5 & 2.02\\
	1/128& 3.638E-6 & 2.02 & 2.135E-3 & 1.00 & 4.235E-6 & 2.01\\
%	1/256& 9.004E-7 & 2.01 & 1.065E-3 & 1.00 & 1.051E-6 & 2.01\\
%	1/512& 2.235E-7 & 2.01 & 5.318E-4 & 1.00 & 2.615E-7 & 2.01\\
	\hline
		\hline
	h & $\|\divvr(\bfp-\bfp_h)\|$ & rate & $\|\divvr(\bfbeta(u-u_h))\|$ & rate & 
	$\| z_h\|$ & rate\\
	\hline
%	1/8  & 9.885E-5 &      & 1.944E-2 &      & 5.268E-6 &\\
%	1/16 & 2.475E-5 & 1.10 & 1.069E-2 & 0.86 & 8.629E-7& 2.61\\
	1/32 & 6.021E-6 & 2.04 & 5.343E-3 & 1.00 & 1.084E-7 & 2.99\\
	1/64 & 1.475E-6 & 2.03 & 2.669E-3 & 1.00 & 1.360E-8 & 3.00\\
	1/128& 3.638E-7 & 2.02 & 1.333E-3 & 1.00 & 1.703E-9 & 3.00\\
%	1/256& 9.004E-8 & 2.41 & 6.667E-4 & 1.00 & 2.132E-10 & 3.00\\
%	1/512& 2.235E-8 & 3.31 & 3.333E-4 & 1.00 & 2.666E-11& 3.00\\
	\hline
	\end{tabular}
\end{center}

\end{table}

\begin{table}[ht!] 
\caption{\footnotesize{Errors and convergence rates for \cref{ex:3} with $\delta =0.01$}}
\label{tab:ex3-b}
\begin{center}
\subfloat[order=1]{
	\begin{tabular}{||c|c c|c c| c c||}
	\hline
	h & $\|u - u_h\|$ & rate & $\|u-u\|_{H^1(\Omega)}$ & rate & 
	$\| \bfp-\bfp_h\|$ & rate\\
	\hline
%	1/8  & 3.261E-1 &      & 8.516E-0 &      & 3.264E-1 &\\
%	1/16 & 1.980E-1 & 0.72 & 8.693E-0 & -0.03& 1.980E-1 & 0.72\\
%	1/32 & 1.142E-1 & 0.79 & 7.805E-0 & 0.16 & 1.142E-1 & 0.79\\
%	1/64 & 5.878E-2 & 0.96 & 5.910E-0 & 0.40 & 5.878E-2 & 0.96 \\
	1/128& 2.616E-2 & 1.17 & 3.801E-0 & 0.64 & 2.615E-2 & 1.17\\
	1/256& 9.421E-3 & 1.47 & 2.012E-0 & 0.91 & 9.421E-3 & 1.47\\
	1/512& 2.515E-3 & 1.91 & 8.461E-1 & 1.25 & 2.515E-3 & 1.91\\
	\hline
	\hline
	h & $\|\divvr(\bfp-\bfp_h)\|$ & rate & $\|\divvr(\bfbeta(u-u_h))\|$ & rate & 
	$\| z_h\|$ & rate\\
	\hline
	1/128& 2.616E-3 & 1.17 & 3.435E-1 & 0.36 & 1.375E-6 & 2.25\\
	1/256& 9.421E-4 & 1.47 & 2.362E-1 & 0.54 & 2.367E-7 & 2.54\\
	1/512& 2.515E-4 & 1.91 & 1.423E-1 & 0.73 & 3.200E-8 & 2.89\\
	\hline
	\end{tabular}
	}
\end{center}

\begin{center}
\subfloat[order =2]{
	\begin{tabular}{||c|c c|c c| c c||}
%	\hline
%	\multicolumn{7}{||c||}{$\delta$ = 0.01 \hspace{2mm} order =2} \\
 	\hline
	h & $\|u - u_h\|$ & rate & $\|u-u\|_{H^1(\Omega)}$ & rate & 
	$\| \bfp-\bfp_h\|$ & rate\\
	\hline
%	1/8  & 1.816E-1 &      & 8.659E-0 &      & 1.817E-1 &\\
%	1/16 & 8.478E-2 & 1.10 & 7.047E-0 & 0.30 & 8.479E-2 & 1.10\\
%	1/32 & 3.941E-2 & 1.11 & 4.905E-0 & 0.52 & 3.941E-2 & 1.10\\
%	1/64 & 1.523E-2 & 1.37 & 2.687E-0 & 0.87 & 1.523E-2 & 1.37 \\
	1/128& 4.470E-3 & 1.77 & 1.123E-0 & 1.25 & 4.470E-3 & 1.77\\
	1/256& 8.402E-4 & 2.41 & 3.103E-1 & 1.86 & 8.402E-4 & 2.41\\
	1/512& 8.442E-5 & 3.31 & 5.018E-2 & 2.63 & 8.441E-5 & 3.32\\
	\hline
	\hline
	h & $\|\divvr(\bfp-\bfp_h)\|$ & rate & $\|\divvr(\bfbeta(u-u_h))\|$ & rate & 
	$\| z_h\|$ & rate\\
	\hline
%	1/8  & 1.816E-2 &      & 5.431E-1 &      & 1.110E-4 &\\
%	1/16 & 8.478E-3 & 0.72 & 3.253E-1 & 0.74 & 1.650E-5& 2.75\\
%	1/32 & 3.941E-3 & 0.79 & 2.414E-1 & 0.43 & 3.650E-6 & 2.17\\
%	1/64 & 1.523E-3 & 0.96 & 1.423E-1 & 0.76 & 6.245E-7 & 2.54\\
	1/128& 4.470E-4 & 1.77 & 6.839E-2 & 1.06 & 8.267E-8 & 2.91\\
	1/256& 8.402E-5 & 2.41 & 2.610E-3 & 1.39 & 7.847E-9 & 3.39\\
	1/512& 8.442E-6 & 3.32 & 7.817E-3 & 1.74 & 5.088E-10& 3.94\\
	\hline
	\end{tabular}
	}
\end{center}

\end{table}

To test the robustness of our method for the pure convection problem, in \cref{fig:osillation-free_2} we show the numerical solutions for Example 3 with $k=1$ and upon $\delta = 0.001$ on 
structured meshes with various mesh sizes.
\textcolor{black}{We observe that, even for the highly sharp internal layer problem on relatively coarse meshes, the numerical solutions show no sighs of global spurious oscillation. When the mesh size does not resolve the layer only mild and localized oscillation presents around the internal layer.}

\begin{figure}[h!]   
    \centering
    \includegraphics[width=0.6\textwidth]{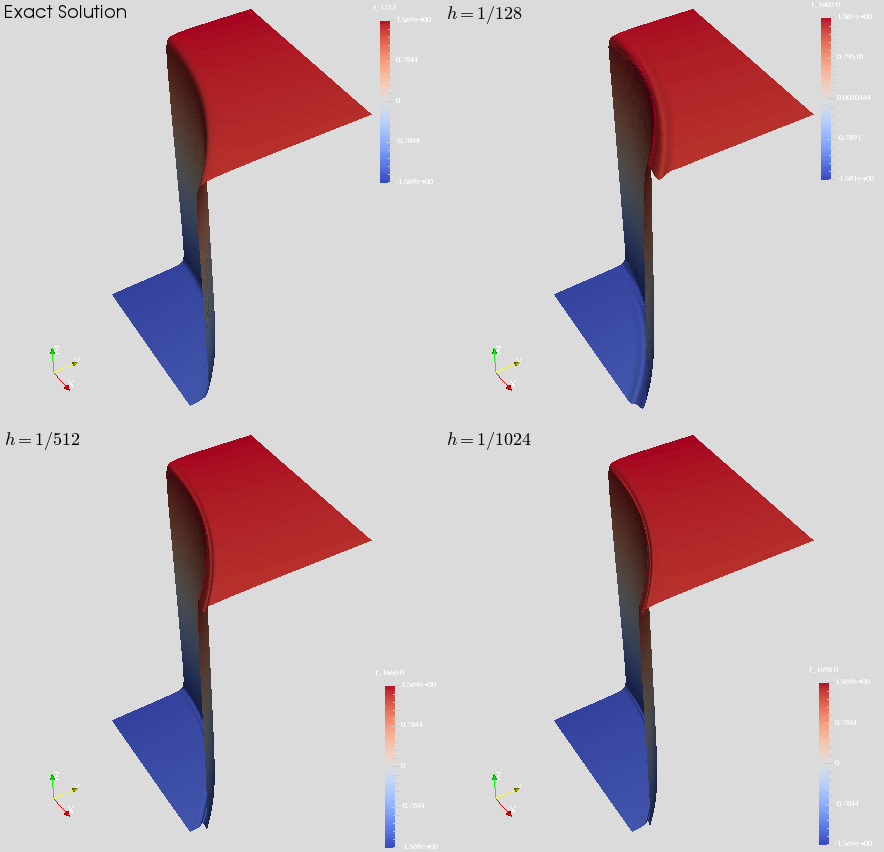}
    \caption{Various numerical solutions for \cref{ex:2} with  $\delta=0.001$}
     \label{fig:osillation-free_2}
\end{figure}

\section{Outflow Boundary Layers}
\label{sec:outflow}
From \cref{fig:osillation} we see that the current method does not
handle outflow boundary well because its lack of upstream
mechanism. In this section we propose two simple modifications of the
method based on the current setting that removes the global spurious
oscillation. More specifically, one method impose the boundary
condition weakly, whereas the other takes the approach of weighting
the stabilizer such that the oscillation is more ``costly" closer to
the inflow boundary, and, hence, introduces a notion of upwind
direction.
\subsection{Weakly imposed boundary conditions}
In this approach we weakly impose the Dirichlet boundary conditions,
giving different weight to the inflow and outflow boundary. 
The modified weak formulation is to find  $(u_h,\bfp_h,z_h) \in V_h^k \times
RT^k \times X_h^k$ such that
\begin{equation} \label{eq:weak boundary}
\mathcal{A}_1[(u_h,\bfp_h,z_h),(v_h,\bfq_h,x_h)] = l_h(x_h),\quad \forall \,
(v_h,\bfq_h,x_h) \in V_h^k \!\times\! RT^k \!\times\! X_h^k,
\end{equation}
where
\begin{equation*}
\begin{split}
\mathcal{A}_1[(u_h,\bfp_h,z_h),(v_h,\bfq_h,x_h)] &= b(\bfq_h, v_h,z_h)+
b(\bfp_h, u_h,x_h)+s[(u_h,\bfp_h),(v_h,\bfq_h)]\\
&+ \left<  (h [\beta \cdot \bfn]_-^2+ \gamma \epsilon^2/h) u, v
\right>_{\partial \Omega} 
\end{split}
\end{equation*}
and
\[
 l_h(x_h) = ( f, x_h)
 + \left<  (h [\beta \cdot \bfn]_-^2+ \gamma \epsilon^2/h) g, v
\right>_{\partial \Omega}.
\]
In the above formulation $[\beta \cdot n]_- = \min(0,\beta \cdot
n)$ and $\epsilon = \min(\lambda_A)$, i.e., the smallest eigenvalue of $A$.
\begin{remark}
Note that in the above method the Dirichlet  boundary condition is
enforced weakly everywhere. Alternatively one may impose the Dirichlet
conditions strongly on the inflow boundary. The outcome
turns out to be similar.
\end{remark}

%$\epsilon = 0.0001$, $\gamma = 0.1$
\begin{figure}[h!]   
    \centering
    \includegraphics[width= 0.6\textwidth]{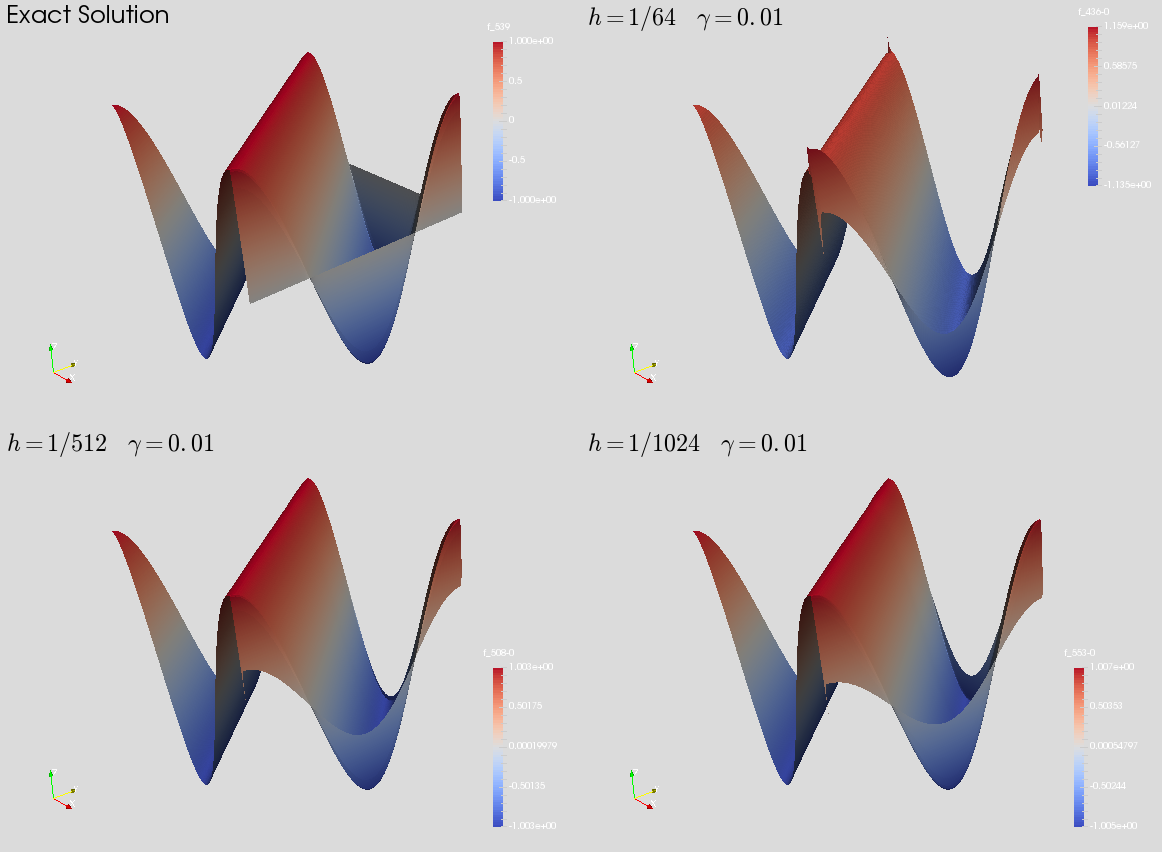}
    \caption{ \footnotesize{Various numerical solutions of \cref{ex:1} with weakly imposed Dirichlet condition}}
     \label{fig:osillation-free_weak_boundary}
\end{figure}

\Cref{fig:osillation-free_weak_boundary} shows the numerical solutions for \cref{ex:1} computed by 
using the method \cref{eq:weak boundary} on the same meshes in \cref{fig:osillation} for $\epsilon = 0.002$.
Comparing to \cref{fig:osillation} the spurious oscillation is
completely removed and only a mild local oscillation is observed for
the coarsest mesh on the outflow boundary layers.

We also test the method on a commonly used benchmark problem with an internal and outflow boundary layers \cite{Burman:2005aa}.
\begin{example} \label{ex:4}
Let $u$ be the solution that satisfies
\begin{equation*}
\begin{split}
	\nabla \cdot ( \bfbeta u - \epsilon \nabla u )  =0& \quad \mbox{on} \,\Omega, \\
	u =1 & \quad \mbox{on} \,\Gamma_{L},\\
	u =0 & \quad \mbox{on} \, \partial \Omega \setminus \Gamma_L\\
\end{split}
\end{equation*}
where $\Omega$ is the unit square domain $(0,1)^2$, $\beta = (1,-0.5)$ and $\Gamma_L$ is the left boundary of the square, i.e., $x=0$. $\epsilon$ is the diffusion coefficient and in out test we choose $\epsilon = 0.001$ in which case the internal and boundary layers are very sharp.

\end{example} 

\textcolor{black}{In \cref{fig:outflow_weak_boundary} we compare the results between the original method (see figures on the top) and the method of 
\cref{eq:weak boundary} (see figures at the bottom). 
We observe that the weak boundary condition method results in an
accurate solution in the bulk, with unresolved layers, that are
resolved as the mesh-size is small enough, whereas the approximation with strongly
imposed conditions has a globally large error.} 

%\begin{figure}[h!]   
%    \centering
%    \includegraphics[width=0.6\textwidth]{./figures/outflow_weak_boundary1}
%    \caption{Comparison between the original method and the weak boundary condition method}
%     \label{fig:outflow_weak_boundary}
%\end{figure}

\begin{figure}
\centering
\subfloat[]{\includegraphics[width=0.3\textwidth]{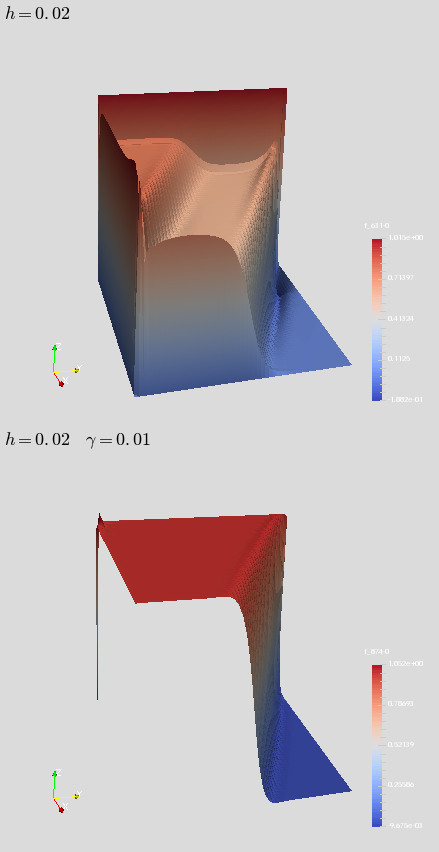}} 
\subfloat[]{\includegraphics[width=0.3\textwidth]{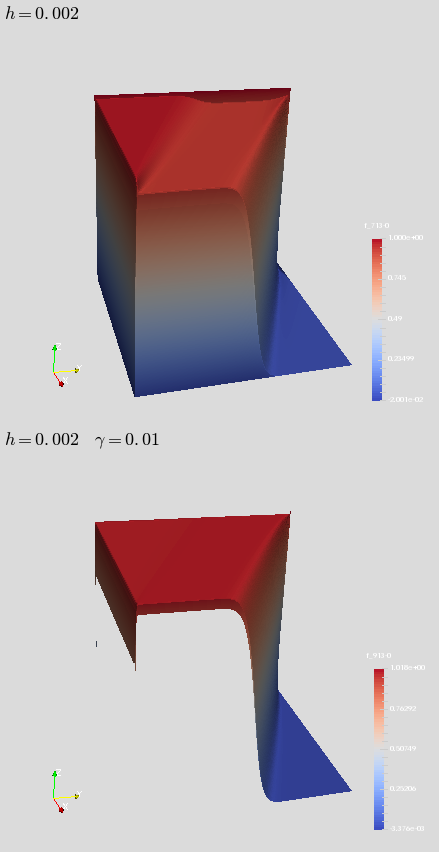}}
\subfloat[]{\includegraphics[width=0.3\textwidth]{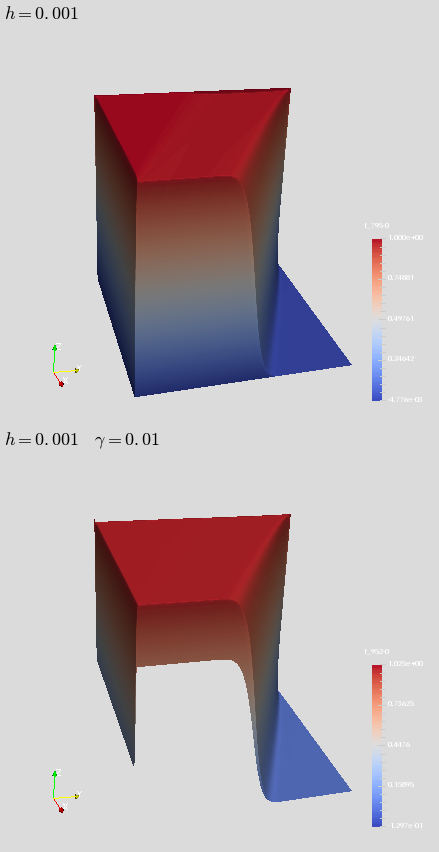}}
\caption{
Numerical performance of the weak boundary method for \cref{ex:4}.} 
\label{fig:outflow_weak_boundary}
\end{figure}

\subsection{Weighted stabilization method}
In this subsection we propose a method where a weight function is
introduced in the stabilizing term $s$.
The motivation here is to change the stabilization making oscillations more
``costly'' closer to the inflow boundary, this way introducing a
notion of upwind direction. More precisely, we introduce a weight function
$\eta: \Omega \rightarrow \mathbb{R}$ such that
\begin{equation}\label{eta}
	\beta > 0 \quad \mbox{and} \quad \bfbeta \cdot (\nabla \eta)< 0.
\end{equation}
For \cref{ex:3} we choose
\[
	\eta = 3 - \bfbeta \cdot (x, y)
\]
and 
for \cref{ex:4} we choose
\[
	\eta = 2 - \bfbeta \cdot (x, y).
\]
It is easy to check that  \cref{eta} holds for both problems. We then
introduce $\eta^p$, for some $p>0$ to be specified, as a weight in $s$.

The finite element setting is then 
to find  $(u_h,\bfp_h,z_h) \in V_{g,D}^k \times
RT^k \times X_h^k$ such that
\begin{equation} \label{eq:eta}
\mathcal{A}_2[(u_h,\bfp_h,z_h),(v_h,\bfq_h,x_h)] = l_h(x_h),\quad  \forall \,
(v_h,\bfq_h,x_h) \in V^k \!\times\! RT^k \!\times\! X_h^k,
\end{equation}
where
\begin{equation*}
\begin{split}
\mathcal{A}_2[(u_h,\bfp_h,z_h),(v_h,\bfq_h,x_h)] &= b(\bfq_h, v_h,z_h)+
b(\bfp_h, u_h,x_h)+ s_\eta[(u_h,\bfp_h),(v_h,\bfq_h)],
\end{split}
\end{equation*}
\[
	s_\eta[(u_h,\bfp_h),(v_h,\bfq_h)] = \left(\eta^p (\bfp + A \nabla u - \bfbeta u),
	(\bfq + A \nabla v - \bfbeta v)\right)
\]
and
\[
 l_h(x_h) = ( f, x_h).
 \]

 \begin{figure}[h!]   
    \centering
    \includegraphics[width=0.6\textwidth]{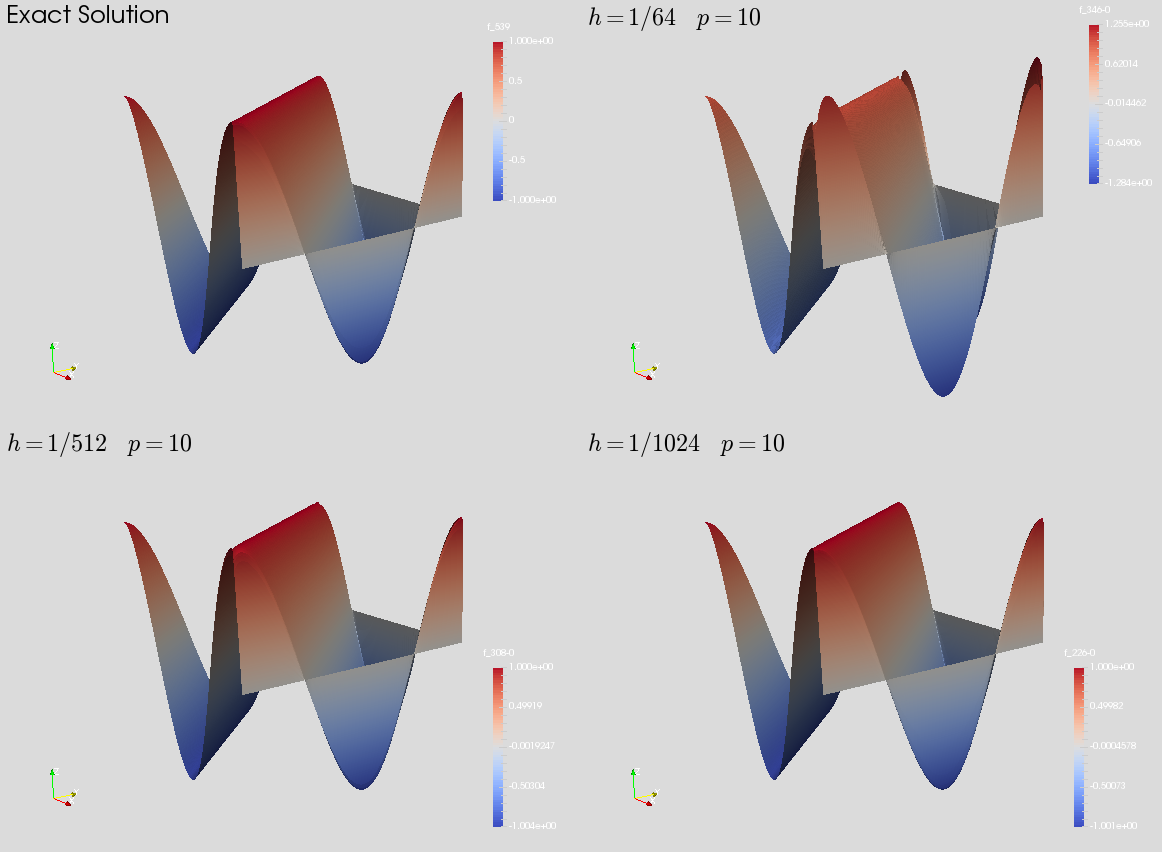}
    \caption{Various numerical solutions with weighted stabilization method for \cref{ex:1}}
     \label{fig:osillation-free_eta}
\end{figure}

 \Cref{fig:osillation-free_eta} shows the numerical solutions for \cref{ex:1} computed by using the method of \cref{eq:eta} on the same meshes in \cref{fig:osillation} for $\epsilon = 0.002$. 
Comparing to \cref{fig:osillation} we observe that the global spurious oscillation has been 
eliminated even for very coarse mesh. Local oscillations along the
outflow boundary does appear when the layer is not completely resolved. 
We also test this method for \cref{ex:4} (see results in \cref{fig:osillation-free_eta1}). 
% \begin{figure}[h!]   
%    \centering
%    \includegraphics[width=0.6\textwidth]{./figures/outflow_eta}
%    \caption{Numerical performance for weighted stabilization method}
%     \label{fig:osillation-free_eta1}
%\end{figure}

\begin{figure}
\centering
\subfloat[]{\includegraphics[width=0.25\textwidth]{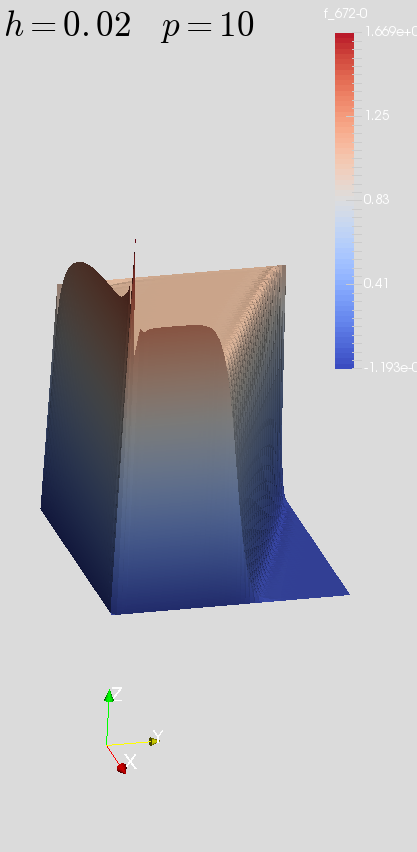}} 
\subfloat[]{\includegraphics[width=0.25\textwidth]{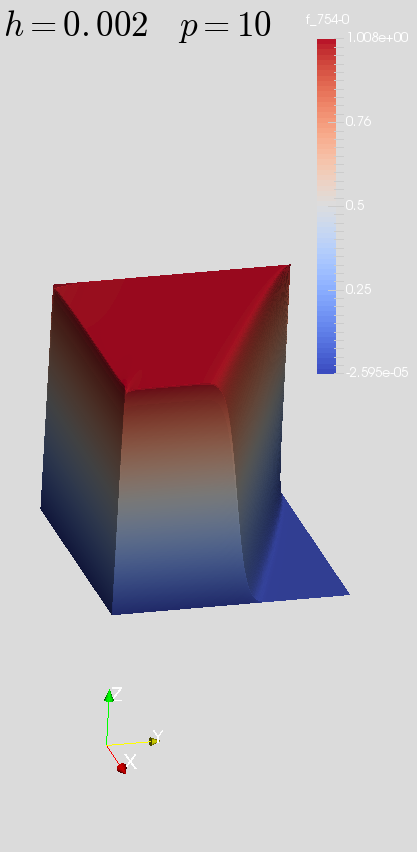}}
\subfloat[]{\includegraphics[width=0.25\textwidth]{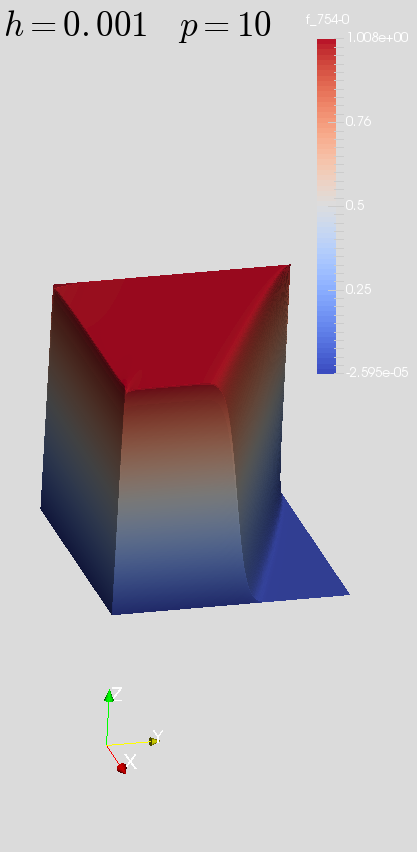}}
\caption{Numerical performance of the weighted stabilization method.} 
\label{fig:osillation-free_eta1} 
\end{figure}

%\section{New Improvement}
%Goal: to improve the convergence result without the coercivity for the convex dominated part so that the all the results in this manuscript applies to indefinite advection--diffusion problem in both diffusion and advection dominated regimes.
%
%If the following holds (needs to prove):
%\[
% \|u_h - i_h u\| \lesssim h^{-1} \| u_h - i_h u\|_{H^{-1}}
%\]
%\[
%\|u-i_h u\|_{H^{-1}(\Omega)} \le h^{k+1} |u|_{H^{k+1}(\Omega)}
%\]
%\[
%	\|u-u_h\|_{H^{-1}(\Omega)} \le h\tnorm{(u - u_h, \bfp - \bfp_h)}
%\]
%then we will have
%\begin{eqnarray*}
%	\| u -u_h\| &\le& \| u - i_h u\| + h^{-1}\| u_h - i_h u\|_{H^{-1}(\Omega)}\\
%	&\lesssim& h^{k+1} |u|_{H^{k+1}(\Omega)} + h^{-1}\|u-u_h\|_{H^{-1}(\Omega)} 
%	+
%	h^{-1}\|u-i_h u\|_{H^{-1}(\Omega)} \\[2mm]
%	&\lesssim&
%	h^{k+1} |u|_{H^{k+1}(\Omega)} + h^{-1}\|u-u_h\|_{H^{-1}(\Omega)}\\[2mm]
%	&\lesssim&
%	h^{k} \left( |u|_{H^{k+1}(\Omega)} +  |\bfp|_{H^{k}(\Omega)}  \right)
%\end{eqnarray*}
%and
%\begin{eqnarray*}
%\| \nabla \cdot (\bfp - \bfp_h)\| &\le& \|\nabla \cdot (\bfp - \bfp_h) + \mu (u- u_h)\| + \|\mu(u-u_h)\|\\
%&\lesssim& h^{k}| f|_{H^k(\Omega)}  + \|\mu(u-u_h)\| \\
%&\lesssim& h^k \left( | f|_{H^k(\Omega)} +  |u|_{H^{k+1}(\Omega)} +  |\bfp|_{H^{k}(\Omega)} \right).
%%\| \nabla \cdot \bfbeta(u - u_h)\| \le 
%%\| \nabla \cdot \bfbeta(u - u_h) - \nabla \cdot(\bfp - \bfp_h)\| + \|\nabla \cdot(\bfp - \bfp_h)\|
%\end{eqnarray*}

%\section*{Acknowledgments}
%We would like to acknowledge the assistance of volunteers in putting
%together this example manuscript and supplement.

\clearpage % force a pagebreak and flush all deferred `table` and `figure` environments

\bibliographystyle{siamplain}
\bibliography{primal-dual-bib}
\end{document}

%% file: shared.tex
\usepackage{lipsum}
\usepackage{amsfonts}
\usepackage{graphicx}
\usepackage{epstopdf}
\usepackage{algorithmic}
\usepackage{amstext,amssymb}
\usepackage{float} % !ht command in the image
\usepackage{amsopn}
\usepackage{array} %special column option for tables
\usepackage[caption=false]{subfig}

\ifpdf
  \DeclareGraphicsExtensions{.eps,.pdf,.png,.jpg}
\else
  \DeclareGraphicsExtensions{.eps}
\fi

\newsiamremark{remark}{Remark}
\newsiamremark{hypothesis}{Hypothesis}
\crefname{hypothesis}{Hypothesis}{Hypotheses}
\newsiamthm{claim}{Claim}
\headers{Primal dual mixed method for advection--diffusion equations}{E. Burman and C. He}

\title{Primal dual mixed finite element
  methods for indefinite advection--diffusion equations
  \thanks{
This work was funded by the EPSRC grant EP/P01576X/1.
}}

\author{Erik Burman
	\thanks{Department of Mathematics, 
	University College London, Gower Street, London, UK--WC1E  6BT, United Kingdom 
  	(\email{e.burman@ucl.ac.uk})}
\and 
	Cuiyu He
	\thanks{Department of Mathematics, University College London, Gower Street, London, 
	UK--WC1E  6BT, United Kingdom
  	(\email{c.he@ucl.ac.uk})
	}}

\newcommand{\bfn}{\boldsymbol n}

\newcommand{\bfx}{\boldsymbol x}
\newcommand{\bfw}{\boldsymbol w}

\newcommand{\bfe}{\boldsymbol e}

\newcommand{\bfp}{\boldsymbol p}
\newcommand{\bfq}{\boldsymbol q}

\newcommand{\bn}{\boldsymbol n}

\newcommand{\bfeta}{\boldsymbol \eta}
\newcommand{\bfbeta}{\boldsymbol \beta}

\newcommand{\bfR}{\boldsymbol R}
\newcommand{\piF}{\pi_{F,l}}

\newcommand{\tn}{|\mspace{-1mu}|\mspace{-1mu}|}

\numberwithin{equation}{section}
\newtheorem{lem}{Lemma}[section]
\newtheorem{prop}{Proposition}[section]
\newtheorem{thm}{Theorem}[section]
\newtheorem{rem}{Remark}[section]
\newtheorem{cor}{Corollary}[section]
\newtheorem{assumption}{Assumption}[section]
\newtheorem{example}{Example}[section]

\newcommand{\jump}[1]{[\![#1]\!]}
\newcommand{\tnorm}[1]{\tn#1\tn}

\newcommand{\cF}{ \mathcal{F}}
\newcommand{\divvr}{ \nabla \cdot}

\newcommand{\cT}{\mathcal{T}}